\documentclass[a4paper]{amsart}
\usepackage{amsmath,amssymb,amsthm}
\usepackage[utf8]{inputenc}
\usepackage[T1]{fontenc}
\usepackage{libertine}
\usepackage[libertine,cmintegrals,cmbraces]{newtxmath}
\usepackage{scalerel}
\usepackage{multicol}
\usepackage{enumitem}
\usepackage{comment} 
\usepackage{microtype}
\usepackage{xcolor}
\definecolor{darkred}{RGB}{139,0,0}
\definecolor{darkblue}{RGB}{0,0,139}
\definecolor{darkgreen}{RGB}{0,100,0}
\usepackage[linktocpage]{hyperref}
\hypersetup{
  colorlinks   = true, 
  urlcolor     = darkred, 
  linkcolor    = darkred, 
  citecolor   = darkgreen}
\usepackage{tikz}
\usepackage{tikz-cd}
\usepackage[capitalise]{cleveref}

\usepackage{stmaryrd}

\setenumerate[1]{label=(\roman*),}
\setitemize[1]{label=\raisebox{0.25ex}{\tiny$\bullet$}}

\tikzset{
  symbol/.style={
    draw=none,
    every to/.append style={
      edge node={node [sloped, allow upside down, auto=false]{$#1$}}}}}

\AtBeginDocument{%
   \def\MR#1{}}

\newtheorem{bigthm}{Theorem}
\newtheorem{bigcor}[bigthm]{Corollary}
\newtheorem{thm}{Theorem}[section]
\newtheorem{lem}[thm]{Lemma}
\newtheorem{prop}[thm]{Proposition}
\newtheorem{cor}[thm]{Corollary}

\newtheorem{conjecture}[thm]{Conjecture}

\theoremstyle{definition}

\theoremstyle{remark}

\newtheorem{rem}[thm]{Remark}
\newtheorem*{nrem}{Remark}

\newcommand{\id}{\ensuremath{\operatorname{id}}}
\newcommand{\im}{\ensuremath{\operatorname{im}}}

\newcommand{\free}{\ensuremath{\operatorname{free}}}
\newcommand{\coker}{\ensuremath{\operatorname{coker}}}

\newcommand{\Bord}{\Omega^{\ensuremath{\operatorname{SO}}}}

\newcommand{\Hom}{\ensuremath{\operatorname{Hom}}}

\newcommand{\MT}{\ensuremath{\mathbf{MT}}}

\newcommand{\Thom}{\ensuremath{\mathbf{Th}}}

\newcommand{\MTSO}{\ensuremath{\mathbf{MTSO}}}

\newcommand{\MSO}{\ensuremath{\mathbf{MSO}}}
\newcommand{\MO}{\ensuremath{\mathbf{MO}}}

\newcommand{\SO}{\ensuremath{\operatorname{SO}}}
\newcommand{\STop}{\ensuremath{\operatorname{STop}}}

\newcommand{\BO}{\ensuremath{\operatorname{BO}}}
\newcommand{\BSO}{\ensuremath{\operatorname{BSO}}}
\newcommand{\BSG}{\ensuremath{\operatorname{BSG}}}
\newcommand{\BTop}{\ensuremath{\operatorname{BTop}}}

\newcommand{\BSPL}{\ensuremath{\operatorname{BSPL}}}
\newcommand{\SPL}{\ensuremath{\operatorname{SPL}}}
\newcommand{\BSTop}{\ensuremath{\operatorname{BSTop}}}

\newcommand{\Diff}{\ensuremath{\operatorname{Diff}^{\scaleobj{0.8}{+}}}}
\newcommand{\Diffuo}{\ensuremath{\operatorname{Diff}}}
\newcommand{\Homeo}{\ensuremath{\operatorname{Homeo}^{\scaleobj{0.8}{+}}}}

\newcommand{\BDiff}{\ensuremath{\operatorname{BDiff}^{\scaleobj{0.8}{+}}}}
\newcommand{\BDiffuo}{\ensuremath{\operatorname{BDiff}}}
\newcommand{\BHomeo}{\ensuremath{\operatorname{BHomeo}^{\scaleobj{0.8}{+}}}}

\newcommand{\Sp}{\ensuremath{\operatorname{Sp}}}
\newcommand{\GL}{\ensuremath{\operatorname{GL}}}

\newcommand{\BSp}{\ensuremath{\operatorname{BSp}}}

\newcommand{\bP}{\ensuremath{\operatorname{bP}}}
\newcommand{\ph}{\ensuremath{\operatorname{ph}}}

\newcommand{\interior}[1]{\ensuremath{\operatorname{int}(#1)}}

\newcommand{\odd}{\ensuremath{\operatorname{odd}}}
\newcommand{\ord}{\ensuremath{\operatorname{ord}}}
\newcommand{\num}{\ensuremath{\operatorname{num}}}
\newcommand{\denom}{\ensuremath{\operatorname{denom}}}

\newcommand{\catsingle}[1]{\ensuremath{\mathcal{#1}}}

\newcommand{\oH}{\ensuremath{\operatorname{H}}}

\newcommand{\bfC}{\ensuremath{\mathbf{C}}}

\newcommand{\bfH}{\ensuremath{\mathbf{H}}}

\newcommand{\bfO}{\ensuremath{\mathbf{O}}}
\newcommand{\bfR}{\ensuremath{\mathbf{R}}}
\newcommand{\bfZ}{\ensuremath{\mathbf{Z}}}
\newcommand{\bfQ}{\ensuremath{\mathbf{Q}}}
\newcommand{\bfS}{\ensuremath{\mathbf{S}}}

\newcommand{\cA}{\ensuremath{\catsingle{A}}}

\newcommand{\cL}{\ensuremath{\catsingle{L}}}

\newcommand{\cS}{\ensuremath{\catsingle{S}}}

\newcommand{\Maps}{\ensuremath{\operatorname{Maps}}}

\newcommand{\ra}{\rightarrow}

\newcommand{\xra}[1]{\xrightarrow{#1}}

\newcommand{\Mod}[1]{\ (\mathrm{mod}\ #1)}

\begin{document}

\title[Characteristic numbers of manifold bundles over surfaces with highly connected fibers]{Characteristic numbers of manifold bundles over surfaces with highly connected fibers}
\begin{abstract} 
We study smooth bundles over surfaces with highly connected almost parallelizable fiber $M$ of even dimension, providing necessary conditions for a manifold to be bordant to the total space of such a bundle and showing that, in most cases, these conditions are also sufficient. Using this, we determine the characteristic numbers realized by total spaces of bundles of this type, deduce divisibility constraints on their signatures and $\hat{A}$-genera, and compute the second integral cohomology of $\BDiff(M)$ up to torsion in terms of generalized Miller--Morita--Mumford classes. We also prove analogous results for topological bundles over surfaces with fiber $M$ and discuss the resulting obstructions to smoothing them.
\end{abstract}

\author{Manuel Krannich}
\email{\href{mailto:krannich@dpmms.cam.ac.uk}{krannich@dpmms.cam.ac.uk}}
\address{Centre for Mathematical Sciences, Wilberforce Road, Cambridge CB3 0WB, UK}

\author{Jens Reinhold} 
\email{\href{mailto:jens.reinhold@uni-muenster.de}{jens.reinhold@uni-muenster.de}}
\address{Mathematisches Institut, Einsteinstr. 62, 48149 M{\"u}nster, Germany}

%\subjclass[2010]{57R20, 57R75, 55R10, 55R40.}

\maketitle

%\vspace{-0.35cm}
%%%%%%%%%
%%%%%%%%%

By work of Chern--Hirzebruch--Serre \cite{ChernHirzebruchSerre}, the signature of a closed oriented manifold is multiplicative in fiber bundles as long as the fundamental group of the base acts trivially on the rational cohomology of the fiber. The necessity of this assumption was illustrated by Kodaira \cite{Kodaira}, Atiyah \cite{Atiyah}, and Hirzebruch \cite{HirzebruchSignature}, who constructed manifolds of nontrivial signature fibering over surfaces, whereupon Meyer \cite{MeyerThesis,Meyer} computed the minimal positive signature arising in this way to be $4$. The divisibility of the signature $\sigma\colon\Bord_*\ra\bfZ$ by $4$ is therefore a necessary condition for a manifold to fiber over a surface up to bordism, which, when combined with the vanishing of a certain Stiefel--Whitney number, is also sufficient \cite[Thm 3]{AlexanderKahn}.

A more refined problem is to decide which manifolds fiber over a surface up to bordism with prescribed $d$-dimensional fiber $M$, or equivalently, to determine the image of the map \[\Bord_2\left(\BDiff(M)\right)\ra\Bord_{d+2},\] defined on the bordism group of smooth oriented $M$-bundles over surfaces, which assigns to a bundle its total space. The main objective of this work is to provide a solution to this problem, and its analogue for topological $M$-bundles, for highly connected, almost parallelizable, even dimensional manifolds $M$, such as $M=\sharp^g(S^n\times S^n)$, satisfying a condition on their \emph{genus}
\[g(M) = \max\{g \geq 0\mid\text{there is a manifold }N\text{ such that }M \cong \sharp^g (S^n \times S^n) \sharp N\}.\] 

In the first part of this work, which focusses on the smooth case, we use parametrized Pontryagin--Thom theory to show that the bordism class of a manifold fibering over a surface with fiber $M$ as above lifts to the bordism group $\Omega^{\langle n \rangle}_{2n+2}$ of highly-connected (i.e., $n$-connected) manifolds, and that this property, together with the divisibility of the signature by $4$, detects such bordism classes for $g(M)\ge5$.

\begin{bigthm}\label{bigtheorem:bordismimage}For an $(n-1)$-connected almost parallelizable $2n$-manifold $M$, the image of
\[\Bord_2 \left(\BDiff(M)\right) \to \Bord_{2n+2}\] is contained in the subgroup
\[\im(\Omega^{\langle n \rangle}_{2n+2} \to  \Bord_{2n+2}) \cap \sigma^{-1} (4 \cdot\bfZ).\] Moreover, equality holds if $g(M)\ge5$. For $2n=2$, requiring $g(M)\ge 3$ is sufficient. 
\end{bigthm}

\begin{nrem}\label{rem:reformulationwithout4}
For $4m\neq4,8,16$, the intersection form of a highly connected almost parallelizable $4m$-manifold is even, so its signature is divisible by $8$, which simplifies the subgroup appearing in \cref{bigtheorem:bordismimage} to $\im(\Omega^{\langle n \rangle}_{2n+2} \to  \Bord_{2n+2})$. This is in contrast to the cases $4m = 4, 8, 16$, in which there exist examples of signature $1$, such as $\bfC P^2$, $\bfH P^2$, and $\bfO P^2$.
\end{nrem}

\begin{nrem}\cref{bigtheorem:bordismimage} was already known for small dimensions. For $2n=2$, it follows from Meyer's aforementioned results, for $2n=4$ the statement is trivial since $\Bord_6$ is trivial, and for $2n=6$, it was observed by Randal-Williams in an unpublished note \cite{RWnote}.\end{nrem}

\subsection*{Highly connected bordism}\cref{section:highly_connected_bordism} is devoted to a closer study of the image of the morphism $\Omega^{\langle n \rangle}_{2n+2} \to  \Bord_{2n+2}$ appearing in \cref{bigtheorem:bordismimage}. 
Note that this morphism factors over the quotient of $\Omega^{\langle n \rangle}_{2n+2}$ by Kervaire--Milnor's \cite{KervaireMilnor} group $\Theta_{2n+2}$ of homotopy spheres. A characteristic class argument (see \cref{section:highly_connected_bordism}) shows furthermore that the morphism is trivial unless $2n+2=4m$, leaving us with the need to understand the image of \begin{equation}\label{equation:bordismrestrictionintro}\Omega^{\langle 2m-1 \rangle}_{4m}/\Theta_{4m} \to  \Bord_{4m}.\end{equation} To this end, we combine Wall's work on the classification of highly connected manifolds \cite{Wall2n,Wall2n+1} with a theorem due to Brumfiel \cite{Brumfiel} and enhancements by Schultz \cite{Schultz} and Stolz \cite{Stolz} to derive a concrete description of $\Omega^{\langle 2m-1 \rangle}_{4m}/\Theta_{4m}$ (see \cref{proposition:kernel}), which, however, depends on one unknown: the order $\ord([\Sigma_Q])$ of the class \[[\Sigma_Q]\in\coker(J)_{4m-1}\] of a certain homotopy sphere $\Sigma_Q\in\Theta_{4m-1}$. Although central to the classification of highly connected manifolds, this class is only known in special cases (see \cref{theorem:AndersonSchultz}); Galatius--Randal-Williams \cite[Conj.\,A]{GRWabelian} conjectured it to be trivial. Nevertheless, our description of $\Omega^{\langle 2m-1 \rangle}_{4m}/\Theta_{4m}$ is explicit enough to compute the Pontryagin numbers, signatures, and $\hat{A}$-genera realized by highly connected manifolds in terms of $\ord([\Sigma_Q])$ (see \cref{proposition:latticeodd} and \ref{proposition:latticeeven}), resulting in various descriptions of the image of \eqref{equation:bordismrestrictionintro} expressed in these invariants.

\begin{nrem}After the completion of this work, the class $[\Sigma_Q]\in\coker(J)_{4m-1}$ was shown to be trivial for $m>62$ in work of Burklund, Hahn, and Senger \cite{BHS}.
\end{nrem}

\subsection*{Divisibility of the signature}
In \cref{section:lattices}, we combine these computations with \cref{bigtheorem:bordismimage} to derive divisibility constraints for characteristic numbers of total spaces of smooth $M$-bundles over surfaces, and to determine these numbers completely for $g(M)\ge5$. To state the outcome for the signature, we remind the reader of the minimal positive signature \[\sigma_m=a_m2^{2m+1}(2^{2m-1}-1)\num\left(\frac{\lvert B_{2m}\rvert}{4m}\right)\] realized by an almost parallelizable $4m$-manifold (see \cite{KervaireMilnorBernoulli}). Here $B_i$ is the $i$th Bernoulli number, and $a_m$ is $1$ if $m$ is even and $2$ otherwise. The $2$-adic valuation is denoted by $\nu_2(-)$.

\begin{bigthm}\label{bigthm:signatures_Ahat}
Let $M$ be a closed, highly connected, almost parallelizable $(4m-2)$-manifold. For a smooth $M$-bundle $\pi\colon E\ra S$ over a closed oriented surface, the signature of $E$ is divisible by
\[\begin{cases}
4&\text{for }m=1,2,4\\
\sigma_m &\text{for }m\neq 1\text{ odd}\\ 
2^{i_m}\gcd(\sigma_{m/2}^2,\sigma_{m})&\text{for }m\neq 2,4\text{ even},\end{cases}\] where \[i_m=\min\left(0,\nu_2(\ord([\Sigma_Q]))-2\nu_2(m)-4+2\nu_2(a_{m/2})\right).\]
For $m\ge2$, the $\hat{A}$-genus of $E$ is integral and divisible by
\[\begin{cases}
2\num(\frac{\lvert B_{2m}\rvert}{4m})&\text{for }m\text{ odd}\\ 
\gcd(\num(\frac{\lvert B_{2m}\rvert}{4m}),\num(\frac{\lvert B_{m}\rvert}{2m})^2)&\text{for }m\text{ even}.\end{cases}\] 
Moreover, if $g(M)\ge5$, then these numbers are realized as signatures and $\hat{A}$-genera of total spaces of bundles of the above type. For $m=1$, requiring $g(M)\ge 3$ is sufficient.
\end{bigthm}

\begin{bigcor}\label{bigcor:signatures}
Let $M$ be a closed, highly connected, almost parallelizable $(4m-2)$-manifold. There exists a smooth $M$-bundle over a closed oriented surface with total space of signature $4$ if and only if $4m=4,8,16$. If $4m\neq4,8,16$, then the signature of the total space of such a bundle is divisible by $2^{2m+2}$ for $m$ odd and by $2^{2m-2\nu_2(m)-3}$ for $m$ even.
\end{bigcor}

\begin{nrem}Rovi \cite{Rovi} identified the non-multiplicativity of the signature modulo $8$ as an Arf--Kervaire invariant. As a consequence of \cref{bigcor:signatures}, her invariant vanishes for bundles over surfaces with highly connected almost parallelizable fibers, except possibly for those with total space of dimension $4,8$, or $16$.
\end{nrem}

\subsection*{Generalized Miller--Morita--Mumford classes}Recall the \emph{vertical tangent bundle} $T_\pi E$ of a smooth oriented bundle $\pi\colon E\ra B$ with closed $d$-dimensional fiber $M$ over an $l$-dimensional base, defined as the kernel $T_\pi E=\ker(d\pi\colon TE\ra\pi^*TB)$ of the differential. The \emph{generalized Miller--Morita--Mumford class} $\kappa_c$ associated to a class $c\in\oH^{*+d}(\BSO;k)$ with coefficients in an abelian group $k$ is obtained by integrating $c(T_\pi E)$ along the fibers,
\[\kappa_c(\pi)=\int_{M}c(T_\pi E)\in\oH^{*}(B;k).\] In the universal case, this gives rise to $\kappa_c\in\oH^{*}(\BDiff(M);k)$. If $B$ is stably parallelizable, the bundles $TE$ and $T_\pi E$ are stably isomorphic, so for $c\in\oH^{d+l}(\BSO;k)$, the two characteristic numbers obtained by integrating $\kappa_c(\pi)\in\oH^{l}(B;k)$ over $B$ and $c(TE)\in\oH^{d+l}(E;k)$ over $E$ coincide. For $B$ a surface, this is expressed in the commutativity of the diagram
\begin{equation}\label{eq:CommutativeDiagramIntroduction}
\begin{tikzcd}[column sep=0.5cm, row sep=0.5cm]
\Bord_{2}\left(\BDiff(M)\right) \arrow[r] \arrow[rd, "\kappa_c",swap] & \Bord_{d+2}  \arrow[d, "c"] \\
 & k.
\end{tikzcd}
\end{equation}
All our results on characteristic numbers of total spaces of bundles over surfaces with a fixed fiber $M$ can thus be expressed in terms of values of classes $\kappa_c\in\oH^2(\BDiff(M);k)$ for various $c$. To exemplify this, note that \cref{bigthm:signatures_Ahat} computes the divisibility of the classes $\kappa_{\cL_{4m}}$ and $\kappa_{\hat{\cA}_{4m}}$ in the torsion free quotient $\oH^2(\BDiff(M);\bfZ)_{\free}$.

Exploiting this alternative viewpoint, we determine an explicit basis of $\oH^2(\BDiff(M);\bfZ)_{\free}$ for most highly connected (i.e.\,$(n-1)$-connected) almost parallelizable $2n$-manifolds $M$ in \cref{section:lattices}. Rationally, this group is for $2n\ge6$ and $g(M)\ge7$ given as 
\begin{equation}\label{equ:computationGRW}
\oH^2\left(\BDiff(M);\mathbf{Q}\right) = \begin{cases}
0  & \text{if } 2n \equiv 0 \Mod{4}\\
\mathbf{Q}\kappa_{p_{(n+1)/2}} & \text{if } 2n \equiv 2 \Mod{8} \\
\mathbf{Q}\kappa_{p_{(n+1)/2}} \oplus \mathbf{Q}\kappa_{p^2_{(n+1)/4}} & \text{if } 2n \equiv 6 \Mod{8},
\end{cases}\end{equation} as a consequence of work of Galatius--Randal-Williams \cite{GRWstable,GRWI}. Having computed the values of these Morita--Miller--Mumford classes on bundles over surfaces enables us to lift this rational basis to a basis of $\oH^2(\BDiff(M);\bfZ)_{\free}$. To state the outcome, we define 
\[j_m=\denom\left(\frac{\lvert B_{2m}\rvert}{4m}\right),\quad a_m=\begin{cases}2&\mbox{for }m\mbox{ odd}\\1&\mbox{for }m\mbox{ even,}\end{cases} \quad \text{ and}\quad \mu_m=\begin{cases}2&\mbox{if }m=1,2\\1&\mbox{else,}\end{cases}\] and fix B{\'e}zout coefficients $c_m$ and $d_m$ for the numerator and denominator of $\frac{\lvert B_{2m}\rvert }{4m}$, i.e.,
\[c_m\num\left(\frac{\lvert B_{2m}\rvert }{4m}\right)+d_m\denom\left(\frac{\lvert B_{2m}\rvert }{4m}\right)=1.\]

\begin{bigthm}\label{bigtheorem:cobordismlattice}For a closed, $(n-1)$-connected, almost parallelizable $2n$-manifold $M$ with $2n\ge6$ and $g(M)\ge7$, the group $\oH^2(\BDiff(M);\mathbf{Z})_{\free}$ is generated by \[\frac{ \kappa_{p_{m}}}{2(2m-1)!j_{m}}\] for $2n\equiv2\Mod{8}$, where $m=(n+1)/2$, and by \begin{align*}\frac{\kappa_{p_{k}^2}}{2\mu_{k} a_{k}^2\ord([\Sigma_Q])(2k-1)!^2}\quad\text{and}\quad\frac{2\kappa_{p_{2k}}-\kappa_{p_{k}^2}}{2(4k-1)!j_{2k}}-\frac{\frac{\lvert B_{2k}\rvert }{4k}\left(c_{2k}\frac{\lvert B_{2k}\rvert }{4k}+2d_{2k}(-1)^{k}\right) \kappa_{p_{k}^2}}{2(2k-1)!^2}\end{align*} for $2n\equiv6\Mod{8}$, where $k=(n+1)/4$.
\end{bigthm}

\begin{nrem}
For $2n\ge6$ and $g(M)\ge5$, the torsion in $\oH^2(\BDiff(M);\bfZ)$ for highly connected almost parallelizable $2n$-manifolds $M$ can be derived from work of Galatius--Randal-Williams \cite{GRWabelian}. It is nontrivial, except when $2n=6$. Consequently, in this case, the basis of $\oH^2(\BDiff(M);\bfZ)_{\free}$ in \cref{bigtheorem:cobordismlattice} is also a basis of $\oH^2(\BDiff(M);\bfZ)$.
 \end{nrem}
 
 \subsection*{Topological bundles}
In the final \cref{TopologicalBundles}, we turn from smooth bundles to their topological counterparts. Using smoothing theory, we show that for manifolds $M$ as above, the bordism groups $\Bord_2(\BDiff(M))$ and $\Bord_2(\BHomeo(M))$ of smooth and topological $M$-bundles over surfaces agree rationally (see \cref{lem:homeodiffrationally}). Furthermore, we complement our calculation of the characteristic numbers of the total spaces of such bundles in the smooth case by carrying out the analogous computation for topological bundles, which in particular implies that $\Bord_2(\BDiff(M))$ and $\Bord_2(\BHomeo(M))$ differ integrally. In contrast to the smooth case, this requires direct arguments involving the construction of explicit bundles, since parametrized Pontryagin--Thom theory is not available for topological manifolds so far. In terms of Miller--Morita--Mumford classes, our computation takes the following form.

\begin{bigthm}\label{bigthm:cobordismlatticeTOP}For a smooth, closed, $(n-1)$-connected, almost parallelizable $2n$-manifold with $2n\ge6$ and $g(M)\ge7$, the group $\oH^2(\BHomeo(M);\mathbf{Z})_{\free}$ is generated by \[\frac{\sigma_m}{8}\cdot\frac{\kappa_{p_{m}}}{2(2m-1)!j_m}\] for $2n\equiv2\Mod{8}$, where $m=(n+1)/2$, and by \begin{align*}\frac{\sigma_k^2}{2^{2\nu_2(k)+6}}\cdot\frac{\kappa_{p_{k}^2}}{2a_k^2(2k-1)!^2}\quad\text{and}\quad \frac{\sigma_{2k}}{8}\cdot\frac{2\kappa_{p_{2k}}-\kappa_{p_k^2}}{2(4k-1)!j_{2k}}+\frac{2^{4k-1}(2^{2k-1}-1)^2(\frac{B_{2k}}{4k})^2\kappa_{p_k^2}}{2(2k-1)!^2}\end{align*} for $2n\equiv6\Mod{8}$, where $k=(n+1)/4$ and $k\neq1,2$.\end{bigthm} 

In conjunction with its smooth analogue \cref{bigtheorem:cobordismlattice}, \cref{bigthm:cobordismlatticeTOP} provides obstructions to smoothing topological $M$-bundles over surfaces up to bordism, and shows that there are many examples that cannot be smoothed. This can already be seen by considering the values of the signature of the total space, specialized to which our computations yield the following, which in particular shows that the growth of the divisibility of the signature established in  \cref{bigthm:signatures_Ahat} and \cref{bigcor:signatures} is a purely smooth phenomenon.

\begin{bigthm}\label{bigthm:signaturestopologically}Let $M$ be a smooth, closed, highly connected, almost parallelizable $(4m-2)$-manifold with $m\neq1,2,4$. The signature of total spaces of topological $M$-bundles over a closed oriented surface is divisible by $8$. For $g(M)\ge4$, there exists such a bundle with signature $8$.
\end{bigthm}

\subsection*{Acknowledgements} We are grateful to S{\o}ren Galatius for various valuable discussions.

MK: I would like to thank S{\o}ren Galatius additionally for his kind hospitality during my stay at Stanford University in the course of which this work originated, and Oscar Randal-Williams for a helpful discussion. I was supported by the Danish National Research Foundation through the Centre for Symmetry and Deformation (DNRF92) and by the European Research Council (ERC) under the European Union's Horizon 2020 research and innovation programme (grant agreement No 682922).

JR: I was partially supported by the E.\,K.\,Potter Stanford Graduate Fellowship and by the NSF grant DMS-1405001, and I am supported by the Deutsche Forschungsgemeinschaft (DFG, German Research Foundation) under Germany's Excellence Strategy EXC 2044 -- 390685587, Mathematics M\"unster: Dynamics--Geometry--Structure.

%%%%%%%%%
%%%%%%%%%

\section{Bordism classes of manifolds fibering over surfaces}\label{section:bordism}
Let $M$ be a smooth, oriented, closed $2n$-manifold. As noted in the introduction, assigning to an $M$-bundle over a surface its total space yields a morphism of the form \begin{equation}\label{equation:totalspacemorphism}\Bord_2(\BDiff(M))\ra\Bord_{2n+2}.\end{equation} This section begins our study of its image. We first provide an alternative proof of the fact that the signature of classes in the image of \eqref{equation:totalspacemorphism} is divisible by $4$---a result originally due to Meyer \cite{MeyerThesis, Meyer}. Our proof is inspired by an argument for a related result by Gollinger \cite[Thm 2.0.10]{Gollinger}. This is followed by a recollection on parametrized Pontryagin--Thom theory, which is then used to prove \cref{bigtheorem:bordismimage}, i.e., to identify the image of \eqref{equation:totalspacemorphism} for highly connected, almost parallelizable manifolds $M$ satisfying an assumption on their genus.

\begin{thm}[Meyer]\label{theorem:signaturedivisiblyby4}
The signature of the total space of a smooth bundle of oriented closed manifolds over a surface is divisible by $4$.
\end{thm}
\begin{proof}Let $\pi\colon E\ra S$ be a bundle as in the statement. The Euler characteristic $\chi(E)$ of $E$ is even, as this holds for oriented surfaces and the Euler characteristic is multiplicative in fiber bundles. Choosing a splitting of the short exact sequence $0\ra T_\pi E\ra TE\ra \pi^*TS\ra0$ of vector bundles defining the vertical tangent bundle $T_\pi E$ and a stable trivialization of $TS$, we obtain a stable isomorphism $TE\cong_s T_\pi E\oplus \varepsilon^2$. As $\chi(E)$ is even, we can form the connected sum of $E$ with products of spheres (not changing the signature) to obtain a manifold $E'$ for which $\chi(E')$ vanishes and which still satisfies $TE'\cong_s T_\pi E\oplus \varepsilon^2$. Now $\chi'(E')$ is trivial, so the Euler class of both bundles $TE'$ and $T_\pi E\oplus \varepsilon^2$ vanish, so they are even unstably isomorphic. Consequently, the manifold $E'$ admits two pointwise linearly independent vector fields, so the classical relation between the signature and the vector field problem (see e.g. \cite[Thm IV.2.7]{LawsonMichelsohn}) implies that the signature of $E'$ is divisible by $4$, as claimed. 
\end{proof}

\subsection{Parametrized Pontryagin--Thom theory}\label{section:parametrizedPT}
Given a fibration $\theta\colon B\rightarrow \BSO(2n)$, a \emph{$\theta$-structure} on an oriented  vector bundle $E\ra\BSO(2n)$ is a lift $E\ra B$ along $\theta$. Define the spectrum $\MT\theta$ as the Thom spectrum $\Thom(-\theta^*\gamma_{2n})$ of the inverse of the pullback of the canonical bundle $\gamma_{2n}$ over $\BSO(2n)$. These spectra---in particular the case $\theta=\id$, commonly denoted as $\MT\theta=\MTSO(2n)$---are natural recipients of parametrized Pontryagin--Thom constructions. More precisely, for a smooth bundle $\pi\colon E\ra W$ with closed  $2n$-dimensional fibers and a $\theta$-structure on its vertical tangent bundle, a parametrized version of the Pontryagin--Thom collapse map gives a canonical homotopy class $W\ra\Omega^\infty\MT\theta$ (see e.g.\,\cite{GMTW}). In particular, for an oriented closed $2n$-manifold $M$, this results in a map of the form\begin{equation}\label{equation:parametrizedPTmap}\BDiff(M)\ra\Omega^{\infty}\MTSO(2n),\end{equation} called the \emph{parametrized Pontryagin--Thom map}. For $M$ a surface, the celebrated theorem of Madsen--Weiss \cite{MadsenWeiss}, combined with classical stability results of Harer \cite{Harer} (improved by Boldsen \cite{Boldsen} and Randal-Williams \cite{RWresolutions}), implies that, depending on the genus of $M$, the map \eqref{equation:parametrizedPTmap} provides a good homological approximation of $\BDiff(M)$. It follows from work by Galatius--Randal-Williams \cite{GRWstable,GRWI,GRWII} that this holds for simply connected manifolds $M$ in higher dimensions as well, after replacing $\Omega^\infty\MTSO(2n)$ with a refinement depending on $M$. To explain their program in the special case needed for our purposes, we assume that $M$ is $(n-1)$-connected and $n$-parallelizable; that is, its tangent bundle $M\ra\BSO(2n)$ admits a $\theta_n$-structure for the $n$-connected cover \[\theta_n\colon\BSO(2n)\langle n\rangle\ra \BSO(2n).\] Fix an embedded disk $D^{2n}\subseteq M$ and note that the orientation on $D^{2n}$ extends uniquely to a $\theta_n$-structure $\ell_{D^{2n}}$ on the tangent bundle of $D^{2n}$. For every smooth $M$-bundle $\pi\colon E\ra W$ with a trivialized $D^{2n}$-subbundle, the $\theta_n$-structure on the vertical tangent bundle of the $D^{2n}$-subbundle induced by $\ell_{D^{2n}}$ extends uniquely to a $\theta_n$-structure on the vertical tangent bundle of $\pi$ by obstruction theory. In the universal case, this results in a map of the form \[\BDiffuo(M,D^{2n})\ra\Omega^{\infty}\MT\theta_n,\] which hits a particular component of $\Omega^\infty\MT\theta_n$, denoted by $\Omega^\infty_M\MT\theta_n$. Here $\Diffuo(M,D^{2n})\subseteq\Diff(M)$ is the subgroup of diffeomorphisms fixing $D^{2n}$ pointwise, so $\BDiffuo(M,D^{2n})$ classifies smooth $M$-bundles with a trivialized $D^{2n}$-subbundle.

\begin{thm}[Boldsen, Galatius--Randal-Williams, Harer, Madsen--Weiss]\label{theorem:parametrizedpontryaginthom} For a closed, $(n-1)$-connected, $n$-parallelizable $2n$-manifold $M$, the parametrized Pontryagin--Thom map \[\BDiffuo(M,D^{2n})\ra\Omega^\infty_M\MT\theta_n\] induces an isomorphism on homology in degrees $3*\le{2g(M)-2}$ if $2n=2$, and for $2*\le{g(M)-3}$ if $2n\ge6$. Furthermore, it induces an epimorphism in degrees $3*\le 2g(M)$ if $2n=2$, and for $2*\le{g(M)-1}$ if $2n\ge6$.
\end{thm}

\subsection{Formal bundles with highly connected fibers}\label{sect:proofofThmA}
To use the above recalled theory to determine the image of the morphism $\Bord_2(\BDiff(M,D^{2n})) \ra\Bord_{2n+2}$ for a closed $(n-1)$-connected, $n$-parallelizable (or, equivalently, almost parallelisable) $2n$-manifold $M$, it is convenient to factor this morphism as a composition
\newcommand{\shortArrow}{%
\text{\ }\parbox{.3cm}{\tikz{\draw[->](0,0)--(.3cm,0);}}\text{\ }
}
\begin{equation}\label{equ:factorisation}\Bord_2(\BDiffuo(M,D^{2n})) \shortArrow\Bord_2(\Omega^\infty\MT\theta_n) \shortArrow\Bord_2(\MT\theta_n) \shortArrow \Bord_2(\Sigma^{-2n}\MO\langle n\rangle) \shortArrow\Bord_{2n+2},\end{equation}where the first morphism is induced by the parametrized Pontryagin--Thom map \eqref{equation:parametrizedPTmap}, the second one by the counit $\Sigma^\infty_+\Omega^\infty\MT\theta_n\ra\MT\theta_n$, the third by the stabilization map $\MT\theta_n\ra\Sigma^{-2n}\MO\langle n\rangle$  to the Thom spectrum of the $n$-connected cover $\BO\langle n\rangle\ra\BO$ (cf.\,\cite[Ch.\,3]{GMTW}), and the fourth by the natural map $\MO\langle n\rangle\ra\MSO$ and the multiplication $\MSO\wedge\MSO\ra\MSO$. Before turning to the proof of \cref{bigtheorem:bordismimage}, we show a preparatory lemma concerning the upper horizontal map in the commutative diagram 
\begin{center}
\begin{tikzcd}[column sep=0.5cm, row sep=0.35cm] 
\pi_2\MT\theta_n\arrow[r,"s_*"]\arrow[d,swap]&\Omega^{\langle n\rangle}_{2n+2}\arrow[d]\\
\Bord_2(\Omega^\infty\MT\theta_n)\arrow[r]&\Bord_{2n+2}.
\end{tikzcd}
\end{center}
induced by the stabilisation map $s\colon \MT\theta_n\ra\Sigma^{-2n}\MO\langle n\rangle$ and the Hurewicz homomorphism.

\begin{lem}\label{lem:pi2lemma}The sequence \[\pi_2\MT\theta_n\xrightarrow{s_*}\Omega^{\langle n\rangle}_{2n+2}\ra\begin{cases}0&\mbox{for }n\neq1,3,7\\\bfZ/4&\mbox{for }n=1,3,7\end{cases}\ra0\] is exact, where the middle arrow is for $n=1,3,7$ given by the signature modulo $4$.
\end{lem}
\begin{proof}Galatius--Randal-Williams \cite[Lem.\,5.2, 5.5, 5.6]{GRWabelian} showed that the cokernel of $s_*$ is isomorphic to $\bfZ/4$ for $n=3,7$ and vanishes for $n\neq1,3,7$, which implies the case $n\neq1,3,7$. The manifolds $\bfC P^2$, $\bfH P^2$, and $\bfO P^2$ have signature $1$ and define classes in $\Omega^{\langle n\rangle}_{2n+2}$ for $n=1,3,7$ respectively, so the second claim follows from showing that the minimal signature realised by classes in $\pi_2\MT\theta_n$ for $n=1,3,7$ is $4$. For $n=1$ this follows from \cite[Thm 2.0.10]{Gollinger}. The cited result also implies that classes in the image of $s_*$ have signature divisible by $4$ for all $n$, so the cases $n=3,7$ follow from the calculation of the cokernel of $s_*$.
\end{proof}

Having done the necessary preparations, we prove 
\cref{bigtheorem:bordismimage}.

\begin{proof}[Proof of \cref{bigtheorem:bordismimage}]Since $\Omega^{\langle 1\rangle}_{4}=\Bord_4$, the first claim follows in the case $n=1$ from \cref{theorem:signaturedivisiblyby4}. In the other cases, \cref{theorem:signaturedivisiblyby4} shows that it suffices to show that every total space $E$ of an oriented $M$-bundle $\pi\colon E\ra S$ over an oriented surface is $n$-connected up to bordism, i.e., lifts to $\Omega^{\langle n \rangle}_{2n+2}$. To see this, note that obstruction theory shows that the space of lifts
\begin{center}
\begin{tikzcd}
*\rar\dar&\BSO\langle n\rangle\dar\\
M\arrow[ur,dashed]\rar["T^sM",swap]&\BSO
\end{tikzcd}
\end{center} is contractible, where the lower horizontal map represents the stable oriented tangent bundle of $M$ and the upper horizontal one is any lift of $T^sM$ over a fixed point $*\in M$. Together with the fact that $\pi$ admits a section (which follows from another application of obstruction theory), this implies that the stable oriented vertical tangent bundle $T^s_\pi E\colon E\ra\BSO$ admits a lift to $\BSO\langle n\rangle$. But $S$ is stably parallelizable and  the vertical tangent bundle $T_\pi E$ is stably equivalent to $TE$, so $E$ is bordant to an $n$-connected manifold by doing surgery.

To show that the image of $\Bord_2 \left(\BDiff(M)\right) \to \Bord_{2n+2}$ actually equals the subgroup of the statement in the presence of high genus, note that by \cref{theorem:parametrizedpontryaginthom}, it suffices to show that the image of $\Bord_2(\Omega^\infty_M\MT\theta_n)\ra \Bord_{2n+2}$ is as claimed. This follows from \cref{lem:pi2lemma}, together with \cref{lemma:hurewiczdoesnotchange} below.
\end{proof}

\begin{rem}The previous proof shows that the conclusion of \cref{bigtheorem:bordismimage} holds equally well for $\BDiffuo(M,D^{2n})$ instead of $\BDiff(M)$.
\end{rem}

\begin{lem}\label{lemma:hurewiczdoesnotchange}Let $\Omega^\infty_\bullet\MT\theta_n\subseteq\Omega^\infty\MT\theta_n$ be a path component. The images of \[\pi_2\MT\theta_n\ra\Bord_{2n+2}\quad\text{and}\quad\Bord_2\left(\Omega^\infty_\bullet\MT\theta_n\right)\ra\Bord_{2n+2}\]  agree.
\end{lem}

\begin{proof}We first show that the image of the second morphism does not depend on the chosen path component. For this, note that for an oriented surface $S$, the composition of the natural map from the group of homotopy classes $[S,\Omega^\infty\MT\theta_n]$ to $\Bord_2(\Omega^\infty\MT\theta_n)$ with the map $\Bord_2(\Omega^\infty\MT\theta_n)\ra\Bord_{2n+2}$ is $\pi_0(\MT\theta_n)$-equivariant, where $\pi_0\MT\theta_n$ acts on the domain in the obvious way and on the codomain via the composition $\pi_0\MT\theta_n\ra\Bord_{2n}\ra\Bord_{2n+2}$, the first map being induced by the stabilization map $\MT\theta_n\ra\Sigma^{-2n}\MO\langle n\rangle$ and the natural map $\MO\langle n\rangle\ra\MSO$, and the second one by taking products with $S$. This equivariance can, for instance, be seen by using the geometric description of $[S,\Omega^\infty\MT\theta_n]$ provided by classical Pontryagin--Thom theory. Since $\Bord_2$ is trivial, the action of $\pi_0\MT\theta_n$ on $\Bord_{2n+2}$ is trivial, which implies that it suffices to show the claim for the unit component $\Omega_0^\infty\MT\theta_n$. By comparing $\MT\theta_n$ to its connected cover and using that the unit $\bfS\ra\MSO$ is $1$-connected, one concludes that the images of $\pi_2\MT\theta_n$ and $\Bord_2(\Omega^\infty_0\MT\theta_n)$ in $\Bord_2(\MT\theta_n)$ agree, which, given the factorization \eqref{equ:factorisation}, implies the claim. 
\end{proof}

%%%%%%%%%
%%%%%%%%%

\section{Bordism classes of highly connected manifolds}\label{section:highly_connected_bordism}
This section is concerned with the image of the natural map \begin{equation}\label{equation:bordismrestriction}\Omega_{2n+2}^{\langle n \rangle} \ra \Bord_{2n+2}\end{equation} from highly connected to oriented bordism, which is by \cref{bigtheorem:bordismimage} closely related to smooth bundles over surfaces with highly connected almost parallelizable fibers. As mentioned in the introduction, this morphism is trivial unless $2n+2=4m$, since smooth oriented bordism classes are detected by Stiefel--Whitney and Pontryagin numbers, and both of these vanish for highly connected $(4m+2)$-manifolds: their cohomology is concentrated in degrees $0$, $2m+1$, and $4m+2$, and $\oH^{\ast}(\BO;\bf Z/2)$ is generated by $w_{2^{i}}$ for $i \geq 0$ as a module over the Steenrod algebra. As homotopy spheres are stably parallelizable, the morphism \eqref{equation:bordismrestriction} moreover factors over the quotient by the group of homotopy spheres $\Omega_{2n+2}^{\langle n \rangle}/\Theta_{2n+2}$, which we describe explicitly in the first part of this section by combining work of Brumfiel, Kervaire--Milnor, Schultz, Stolz, and Wall. Using this, we determine the image of \eqref{equation:bordismrestriction} in terms of characteristic numbers and to derive divisibility constraints for the signature and the $\hat{A}$-genera of highly connected manifolds. Throughout this section, we implicitly require all manifolds to be smooth and exclude the case $m=1$.

\subsection{Wall's classification of highly connected almost closed manifolds}
A compact manifold is \emph{almost closed} if its boundary is a homotopy sphere. We denote by $A^{\langle 2m-1\rangle}_{4m}$ the group of oriented, almost closed, $(2m-1)$-connected $4m$-manifolds, up to $(2m-1)$-connected oriented bordism restricting to an $h$-cobordism on the boundary; the group structure is induced by boundary connected sum. This group receives a map from $\Omega^{\langle 2m-1\rangle}_{4m}$ given by cutting out the interior of an embedded $4m$-disk. This fits into an exact sequence 
\begin{equation}\label{equation:WallSES}
\Theta_{4m}\ra\Omega^{\langle 2m-1\rangle}_{4m}\ra A^{\langle 2m-1\rangle}_{4m}\xra{\partial}\Theta_{4m-1}.
\end{equation} 
The first morphism maps a homotopy sphere to its bordism class, and the last one assigns to an almost closed manifold its boundary. From his pioneering work on the classification of highly connected manifolds, Wall \cite{Wall2n} derived a complete description of the groups $A^{\langle 2m-1 \rangle}_{4m}$. For us, the outcome is most conveniently phrased in terms of two particular manifolds $P$ and $Q$, which play a central role in the remainder of this section.

\begin{description}[labelindent=0cm,leftmargin=0cm,itemsep=1ex,topsep=1ex,font=\normalfont\itshape]
\item[{Milnor's $E_8$-plumbing $P$.}]We denote by $P$ \emph{Milnor's $E_8$-plumbing}, i.e., the parallelizable manifold of dimension $4m$ that arises from plumbing eight copies of the disk bundle of the tangent bundle of the $2m$-sphere so that the intersection form of $P$ is isomorphic to the $E_8$-form (see e.g.\,\cite[Ch.\,V.2]{Browder}). Since the latter has signature $8$, the manifold $P$ does as well.

\item[{The plumbing $Q$.}]Let $Q$ be the plumbing of two copies of a $2m$-dimensional linear disk bundle over the $2m$-sphere that generates the image of $\pi_{2m}\BSO(2m-1)$ in $\pi_{2m}\BSO(2m)$. This bundle can be characterized equivalently as having vanishing Euler number and representing a generator of $\pi_{2m}\BSO$ for $m\neq2,4$, and twice a generator for $m=2,4$ (cf.\,\cite[§1A)]{Levine}). Via the isomorphism $\oH^{2m}(S^{2m})\oplus\oH^{2m}(S^{2m})\cong \oH^{2m}(Q)$ induced by the inclusion of the $2m$-skeleton, the intersection form of $Q$ is given by $\bigl(\begin{smallmatrix}0 & 1\\ 1 & 0\end{smallmatrix}\bigr)$, so it has vanishing signature. For $m=2k$ even, the $k$th Pontryagin class of a generator of $\pi_{2m}\BSO$ is $a_k(2k-1)!1^*\in\oH^{4k}(S^{4k})$ (see e.g.\,\cite[Thm 3.8]{Levine}), where $a_k$ equals $1$ for $k$ even and $2$ otherwise, and $1^*$ denotes the Poincar\'e dual to $1\in\oH_0(S^{4k})$. We then compute \begin{equation}\label{equation:pontryaginclass}p_k^2(Q,\partial Q)=2\lambda_k^2a_k^2(2k-1)!^2\cdot1^*\in\oH^{8k}(Q,\partial Q),\end{equation} with $\lambda_k$ being $1$ if $k\neq1,2$ and $2$ elsewise.
\end{description}

The boundaries of both plumbings $P$ and $Q$ are homotopy spheres (see e.g.\,\cite[V.2.7]{Browder}). We denote them by $\Sigma_P$ and $\Sigma_Q$, respectively.

\begin{thm}[Wall]\label{theorem:Wall} The bordism group $A^{\langle 2m-1\rangle}_{4m}$ satisfies
\[A^{\langle 2m-1\rangle}_{4m}\cong\begin{cases}
\bfZ\oplus \bfZ/2&\mbox{if }m\equiv1\Mod{4}\\
\bfZ&\mbox{if }m\equiv3\Mod{4} \\
\bfZ\oplus\bfZ&\mbox{if }m\equiv0\Mod{2},
\end{cases}\] where the first summand is generated by $P$, except for $m=2,4$ where it is generated by $\bfH P^2$ and $\bfO P^2$, respectively. The second summand in the cases $m\not{equiv}3\Mod{4}$ is generated by $Q$.
\end{thm}

\begin{proof}
The statement regarding the isomorphism type of the group  $A^{\langle 2m-1\rangle}_{4m}$ follows from \cite[Thm 2]{Wall2n} and \cite[Thm 11]{Wall2n+1}. Denoting by $\langle Q\rangle\subseteq A^{\langle 2m-1\rangle}_{4m}$ the subgroup generated by $Q$, it follows from the discussion in \cite[p.\,295]{Wall2n+1} that there is an exact sequence
\[0\ra\langle Q\rangle\ra A^{\langle 2m-1\rangle}_{4m}\ra\bfZ\] whose last morphism is induced by the signature. As $\bfH P^2$ and $\bfO P^2$ have signature $1$, the cases $m=2,4$ follow. The other cases are concluded by observing that the intersection form associated to a manifold representing a class in $A^{\langle 2m-1\rangle}_{4m}$ is even for $m\neq 2,4$, so has signature divisible by $8$---the signature of the $E_8$-plumbing.\end{proof}

To treat the different cases for $m$ even in a uniform manner, it is convenient to use a basis of $A^{\langle 2m-1\rangle}_{4m}$ that is different from the one described in the previous theorem. 

\begin{lem}\label{lemma:alternativebasis}\ 
\begin{enumerate}
\item Two almost closed $8k$-manifolds $M$ and $N$ define the same class in $A_{8k}^{\langle 4k-1\rangle}$ if and only if $\sigma(M)=\sigma(N)$ and $p_k^2(M)=p_k^2(N)$.
\item The classes $8\cdot \bfH P^2$ and $8\cdot \bfO P^2$ in $A_{8k}^{\langle 4k-1\rangle}$ equal $P+Q$ for $k=1,2$, respectively.
\item The group $A_{8k}^{\langle 4k-1\rangle}\cong \bfZ\oplus \bfZ$ is generated by $P$ and $Q$ for $k\neq1,2$, by $P$ and $\bfH P^2$ for $k=1$, and by $P$ and $\bfO P^2$ for $k=2$.
\end{enumerate}
\end{lem}

\begin{proof}As opposed to the plumbing $Q$, the manifolds $\bfH P^2$, $\bfO P^2$, and $P$ have nontrivial signature. Since we computed the Pontryagin number $p_{k}^2(Q)$ to be nontrivial in \eqref{equation:pontryaginclass}, claim (i) follows from \cref{theorem:Wall}, as $A_{8k}^{\langle 4k-1\rangle}$ is free abelian of rank $2$. The Pontryagin numbers $p_k^2(\bfH P^2)$ and $p_k^2(\bfO P^2)$ can be computed as $a_{k}^2(2k-1)!^2$ for $k=1,2$, respectively, which agrees with $1/8\cdot p_k^2(Q)$ by \eqref{equation:pontryaginclass}. Claim (ii) follows from (i), remembering that $\sigma(P)=8$ and $\sigma(Q)=0$, and, using \cref{theorem:Wall}, (iii) is a consequence of (ii).
\end{proof}

\subsection{Homotopy $(4m-1)$-spheres}
Recall from \cite{KervaireMilnor} that the group $\Theta_{4m-1}$ of $h$-cobordism classes of oriented homotopy spheres fits into an exact sequence \begin{equation}\label{equation:KervaireMilnorSES}0\ra\bP_{4m}\ra\Theta_{4m-1}\ra\coker(J)_{4m-1}\ra 0\end{equation} involving the subgroup $\bP_{4m}\subseteq \Theta_{4m-1}$ of homotopy spheres bounding parallelizable manifolds and the cokernel of the stable $J$-homomorphism in degree $(4m-1)$. The subgroup $\bP_{4m}$ is generated by the \emph{Milnor sphere} $\Sigma_P=\partial P$. It is of order $\sigma_m/8$ with
\[\sigma_m=a_m2^{2m+1}(2^{2m-1}-1)\num\left(\frac{\lvert B_{2m}\rvert }{4m}\right)\] as defined in the introduction (see e.g.\,\cite[Cor.\,3.20, Lem.\,3.5 (2)]{Levine}). Brumfiel \cite{Brumfiel} has shown that every homotopy sphere $\Sigma\in\Theta_{4m-1}$ bounds a spin manifold $W_\Sigma$ with vanishing decomposable Pontryagin numbers and that the signature $\sigma(W_\Sigma)$ of such a manifold is divisible by $8$. This enabled him to establish a decomposition
 \begin{equation}\label{equation:decomposition}\Theta_{4m-1}\cong\bP_{4m}\oplus\coker(J)_{4m-1}\end{equation} via a splitting $s_B\colon\Theta_{4m-1}\ra\bP_{4m}$ of the exact sequence \eqref{equation:KervaireMilnorSES}, defined by mapping a homotopy sphere $\Sigma$ to $(\sigma(W_\Sigma)/8)\cdot \Sigma_P$. Refining Brumfiel's definition, Stolz \cite{Stolz} gave a formula for $s_B(\Sigma)$ in terms of invariants of any spin manifold that bounds $\Sigma$, without assumptions on its characteristic numbers. To state his result, we fix integers $c_m$ and $d_m$ with 
\[c_m\num\left(\frac{\lvert B_{2m}\rvert }{4m}\right)+d_m\denom\left(\frac{\lvert B_{2m}\rvert }{4m}\right)=1,\] and define a rational polynomial $\cS_m\in\oH^{4m}(\BSO;\bfQ)$ in Pontryagin classes as
\[\cS_m=\cL_m+\frac{\sigma_m}{a_m}\left(c_m\hat{\cA}_m+(-1)^md_m(\hat{\cA}\ph)_m\right),\] involving the $\cL$- and $\hat{\cA}$-class, as well as product of the $\hat{\cA}$-class with the reduced Pontryagin character $\ph$. Here \emph{reduced} refers to the triviality of $\ph$ in degree $0$. The polynomial $\cS_m$ has no contributions from the $m$th Pontryagin class (see \cite[p.\,2]{Stolz}), so its evaluation on an oriented almost closed manifold $M$ can be considered as a relative class $\cS_m(M)\in\oH^{4m}(M,\partial M;\bfQ)$. 
\begin{thm}[Stolz]\label{theorem:Stolz}For an almost closed spin manifold $M$ of dimension $4m$, the invariant
\[s(M)=\frac{1}{8}\Big(\sigma(M)-\langle\cS_m(M),[M,\partial M]\rangle\Big)\] is integral and computes the value of Brumfiel's splitting on the boundary of $M$, i.e.,
\[s_B(\partial M)=s(M)\cdot \Sigma_P.\]
\end{thm}

\subsection{Closing almost closed manifolds}\label{section:closingalmostclosedmanifolds}
By the exactness of the sequence \eqref{equation:WallSES}, the bordism group $\Omega_{4m}^{\langle 2m-1\rangle}/\Theta_{4m}$ is naturally isomorphic to the kernel of the morphism \[\partial\colon A_{4m}^{\langle 2m-1\rangle}\ra\Theta_{4m-1},\] which leads us to identify these two groups henceforth. Since $A_{4m}^{\langle 2m-1\rangle}$ is generated by the classes of $P$ and $Q$ for $m\neq2,4$ by \cref{theorem:Wall}, we need to examine their boundaries $\Sigma_P$ and $\Sigma_Q$ in $\Theta_{4m-1}$ in order to determine the kernel in question. As mentioned earlier, the Milnor sphere $\Sigma_P$ is well understood and generates the subgroup $\bP_{4m}$. Regarding $\Sigma_Q$, we use Stolz's invariant to compute its image under the projection onto $\bP_{4m}$ with respect to the decomposition $\eqref{equation:decomposition}$ (see \cref{lemma:s(Q)}). Concerning its image $[\Sigma_Q]$ in $\coker(J)_{4m-1}$, there are partial results by Anderson and Schultz. To state the ones relevant for us, denote by \[j_n=\denom\left(\frac{\lvert B_{2n}\rvert }{4n}\right)\] the size of the image of the stable $J$-homomorphism in degree $4n-1$ (see \cite{Adams,Quillen}).

\begin{thm}[Anderson, Schultz]\label{theorem:AndersonSchultz}The class $[\Sigma_Q]$ in $\coker(J)_{4m-1}$ satisfies
\begin{enumerate}
\item $j_{m/2}^2\cdot[\Sigma_Q]=0$ for $m\neq 2,4$ even, and
\item $[\Sigma_Q]=0$ for $m$ odd.
\end{enumerate}
\end{thm}
\begin{proof} Claim (i) follows from the beginning of the proof of \cite[Lem.\,1.5]{Anderson}. To prove claim (ii), observe that $\Sigma_Q$ is trivial if $m\equiv 3\Mod{4}$ by \cref{theorem:Wall}. A result by Schultz \cite[Cor.\,3.2, Thm 3.4 iii)]{Schultz} settles the remaining case.
\end{proof}

Despite Anderson's bound on its order, very little is known about the class $[\Sigma_Q]\in\coker(J)_{4m-1}$ for $m$ even. In the first two cases $m=2,4$, it is known to be trivial (cf.\,\cite[Ch.\,6]{GRWabelian}) and Galatius--Randal-Williams \cite[Conj.\,A]{GRWabelian} conjectured that this is always the case. A weaker version of this conjecture appeared independently in work of Bowden--Crowley--Stipsicz \cite[Conj.\,5.9]{BowdenCrowleyStipsicz}.

\begin{conjecture}[Galatius--Randal-Williams]\label{conjecture:annoyingsphere} For $m$ even, $[\Sigma_Q]$ is trivial in $\coker(J)_{4m-1}$.
\end{conjecture}

\begin{rem}After the completion of this work, Burklund, Hahn, and Senger \cite{BHS} established \cref{conjecture:annoyingsphere} for all even $m>62$.
\end{rem}

To state our formula for the image of $\Sigma_Q$ under the projection onto $\bP_{4m}$, we denote by 
\[T_n=2^{2n}(2^{2n}-1)\frac{\lvert B_{2n}\rvert }{2n}\] the $n$th \emph{tangent number}, which is known to be an integer (see e.g.\,\cite[Rem.\,1.18]{BernoulliZeta}).

\begin{lem}\label{lemma:s(Q)}Stolz' invariant $s(Q)$ of $Q$ vanishes if $m$ is odd. For $m=2k$ even, it satisfies
\[s(Q)=-\frac{\lambda_k^2}{8j_k^2} \left( \sigma_k^2+a_k^2\sigma_{2k}\num\left(\frac{\lvert B_{2k}\rvert }{4k}\right)\left(c_{2k}\num\left(\frac{\lvert B_{2k}\rvert }{4k}\right)+2(-1)^kd_{2k}j_k\right)\right),\] 
as well as
\[s(Q)=\frac{\lambda_k^2a_k^2}{4}\left(\sigma_{2k}d_{2k}\frac{\lvert B_{2k}\rvert }{4k}\left(\frac{\lvert B_{2k}\rvert }{\lvert B_{4k}\rvert }+(-1)^{k+1}\right)-\frac{T_k^2}{4}\right),\] where $\lambda_k=1$ if $k\neq 1,2$ and $\lambda_k=2$ otherwise.\end{lem}

We postpone the proof of this lemma to the next subsection and continue by elaborating on some of its consequences instead.

\begin{cor}\label{corollary:s(Q)consequences}For $m$ even, the homotopy spheres $\Sigma_P$ and $\Sigma_Q$ in $\Theta_{4m-1}$ satisfy
\begin{enumerate}
\item $j_{m/2}^2\cdot\Sigma_Q=(\sigma_{m/2}^2/8)\cdot \Sigma_P$ for $m\neq2,4$, and
\item $\Sigma_Q=-\Sigma_P$ for $m=2,4$.
\end{enumerate}
\end{cor}
\begin{proof}We write $m=2k$. By \cref{theorem:AndersonSchultz}, the homotopy sphere $j_{k}^2\cdot\Sigma_Q$ lies in $\bP_{8k}$ for $2k\neq2,4$, so \cref{theorem:Stolz} gives the relation $j_{k}^2\cdot\Sigma_Q=j_{k}^2s(Q)\cdot\Sigma_P$. From the first formula of \cref{lemma:s(Q)}, we see that $j_{k}^2s(Q)$ is congruent to $\sigma_{k}^2/8$ modulo $\sigma_{2k}$. This has item (i) as a consequence. Since $\num(\lvert B_4\rvert/8)=\num(\lvert B_8\rvert/16)=1$, we conclude $d_2=d_4=0$. Thus, the second formula of the lemma gives $s(Q)=-T_1^2$ for $k=1$ and $s(Q)=-T_2^2/4$ for $k=2$. As $T_1=1$ and $T_2=2$, we have $s(Q)=-1$ for $k=1,2$, which, together with \cref{theorem:Stolz} and the fact that $[\Sigma_Q]\in\coker(J)_{8k-1}$ vanishes for $k=1,2$, implies item (ii).
\end{proof}

Combining Wall's classification, Stolz' invariant, and \cref{theorem:AndersonSchultz}, we determine the kernel of the boundary map in the sequence \eqref{equation:WallSES}, and hence the bordism group $\Omega_{4m}^{\langle 2m-1\rangle}/\Theta_{4m}$, as follows. The order of the class $[\Sigma_Q]$ in $\coker(J)_{4m-1}$ is denoted by $\ord([\Sigma_Q])$.

\begin{thm}\label{proposition:kernel}The bordism group $\Omega_{4m}^{\langle 2m-1\rangle}/\Theta_{4m}$ satisfies
\[\Omega_{4m}^{\langle 2m-1\rangle}/\Theta_{4m}\cong \begin{cases}
\bfZ\oplus \bfZ/2&\mbox{if }m\equiv 1\Mod{4}\\
\bfZ&\mbox{if }m\equiv 3\Mod{4}\\
\bfZ\oplus \bfZ&\mbox{if }m\equiv 0\Mod{2}.
\end{cases}\] The first summand is generated by $(\sigma_m/8)\cdot P$ in all cases and the $\bfZ/2$-summand for $m\equiv 1\Mod{4}$ is generated by $Q$.
For $m$ even, the second summand is generated by $\bfH P^2$ if $m=2$, by $\bfO P^2$ if $m=4$, and by $\ord([\Sigma_Q])\big(Q-s(Q)\cdot P\big)$ otherwise.
\end{thm}

\begin{proof}
By \cref{theorem:Wall} and \cref{lemma:alternativebasis}, the group $A_{4m}^{\langle 2m-1\rangle}$ is isomorphic to a direct sum $\bfZ\oplus C$ for a cyclic group $C$, where the first summand is generated by $P$, and the second summand by $Q$ for $m\neq2,4$, by $\bfH P^2$ for $m=2$, and by $\bfO P^2$ for $m=4$. Recall that the Milnor sphere $\Sigma_P$ generates the cyclic group $\bP_{4m}$, which is of order $\sigma_{m}/8$. By exactness of \eqref{equation:KervaireMilnorSES}, the homotopy sphere $\ord([\Sigma_Q])\cdot\Sigma_Q$ is contained in $\bP_{4m}$, so it coincides with $\ord([\Sigma_Q])s(Q)\cdot\Sigma_P$ by \cref{theorem:Stolz}. Using this, it follows from elementary algebraic considerations that the classes $(\sigma_m/8)\cdot P$ and $\ord([\Sigma_Q])\big(Q-s(Q)\cdot P\big)$ generate the kernel for $m\neq2,4$. As $Q$ has infinite order for $m$ even, this settles the case for $m\neq 2,4$ even. The class of $Q$ has order $2$ for $m\equiv 1\Mod{4}$ and is trivial for $m\equiv 3\Mod{4}$ by \cref{theorem:Wall}. Together with the fact that, for $m$ odd, we have $s(Q)=0$ by \cref{lemma:s(Q)} and $\ord([\Sigma_Q])=1$ by \cref{theorem:AndersonSchultz}, this concludes the proof for $m$ odd. The cases $m=2,4$ are immediate.
\end{proof}

\begin{rem}Note that, for $m=2k\neq2,4$, \cref{corollary:s(Q)consequences} implies that $(\sigma_{k}^2/8)\cdot P-j_{k}^2\cdot Q$ is contained in the kernel under consideration.
\end{rem}

\begin{rem}\label{remark:usefulbasis}After replacing $\bfH P^2$ by $4\cdot \bfH P^2$ and $\bfO P^2$ by $4\cdot \bfO P^2$, the statement of \cref{proposition:kernel} remains valid if $\Omega^{\langle 2m-1\rangle}_{4m}/\Theta_{4m}$ is replaced by the subgroup $\Omega^{\langle 2m-1\rangle}_{4m}/\Theta_{4m} \cap \sigma^{-1} (4\cdot \bfZ)$. This is because the signatures of $P$ and $Q$ are divisible by $8$, whereas $\sigma(\bfH P^2)=\sigma(\bfO P^2)=1$.
\end{rem}

\subsection{Characteristic numbers of highly connected manifolds}\label{section:charnumberformulas}
This subsection serves to prove \cref{lemma:s(Q)} and to use it in combination with \cref{proposition:kernel} to compute the lattices of characteristic numbers realized by closed $(2m-1)$-connected $4m$-manifolds. Note that for such manifolds, all Pontryagin classes vanish, except possibly $p_{m}$, and $p_{m/2}$ if $m$ is even. As a result of this, the $\cL$-class, the $\hat{\cA}$-class, the reduced Pontryagin character $\ph$, and the product $(\hat{\cA}\ph)$ of the latter two, have the form (cf.\,\cite[Ch.\,1,3]{Hirzebruch})
\[\cL_m=\begin{cases}s_mp_m&\mbox{if }m\text{ is odd}\\
\frac{1}{2}(s_{k}^2-s_{2k})p_{k}^2+s_{2k}p_{2k}&\mbox{if }m=2k\text{ is even},
\end{cases}\]
\[\hat{\cA}_m=\begin{cases}\hat{s}_mp_m&\mbox{if }m\text{ is odd}\\
\frac{1}{2}(\hat{s}_{k}^2-\hat{s}_{2k})p_{k}^2+\hat{s}_{2k}p_{2k}&\mbox{if }m=2k\text{ is even},
\end{cases}\]
\[\ph_m=\begin{cases}\frac{(-1)^{m+1}}{(2m-1)!}p_m&\mbox{if }m\text{ is odd}\\
\frac{1}{2(4k-1)!}p_k^2-\frac{1}{(4k-1)!}p_{2k}&\mbox{if }m=2k\text{ is even},
\end{cases}\text{ and}\]
\[(\hat{\cA}\ph)_m=\begin{cases}\frac{(-1)^{m+1}}{(2m-1)!}p_m&\mbox{if }m\text{ is odd}\\
\frac{(-1)^{k+1}\hat{s}_k}{(2k-1)!}p_k^2+\frac{1}{2(4k-1)!}p_k^2-\frac{1}{(4k-1)!}p_{2k}&\mbox{if }m=2k\text{ is even},
\end{cases}\]
where $\hat{s}_n$ and $s_n$ are given by
\begin{equation}\label{equation:classcoefficients}\hat{s}_n=\frac{-1}{(2n-1)!}\frac{\lvert B_{2n}\rvert }{4n}\quad\text{ and }\quad s_n=-2^{2n+1}(2^{2n-1}-1)\hat{s}_n=\frac{\sigma_n}{a_n(2n-1)!j_n}\end{equation} for $n\ge1$. Solving $\cL_{2k}$ for $p_{2k}$ and expressing $\hat{\cA}_{2k}$ in terms of $\cL_{2k}$ and $p_k^2$, we obtain
\begin{equation}\label{equation:p2k}{p}_{2k}=\frac{1}{s_{2k}}\cL_{2k}-\frac{s_k^2-s_{2k}}{2s_{2k}}p_{k}^2\quad\text{and}\quad\hat{\cA}_{2k}=\frac{s_{2k}\hat{s}_{k}^2-\hat{s}_{2k}s_k^2}{2s_{2k}}p_{k}^2+\frac{\hat{s}_{2k}}{s_{2k}}\cL_{2k}.\end{equation} Substituting the variables with their values, the last equation becomes 
\begin{equation}\label{equation:Ahat}\hat{\cA}_{2k}=\frac{T_k^2}{(2k-1)!^22^{4k+3}(2^{4k-1}-1)}p_k^2-\frac{1}{2^{4k+1}(2^{4k-1}-1)}\cL_{2k}.\end{equation}

\begin{proof}[Proof of \cref{lemma:s(Q)}]
As $Q$ is $(2m-1)$-connected, all its decomposable Pontryagin numbers vanish for $m$ odd, and hence so does the characteristic number $\langle\cS_m(Q),[Q,\partial Q]\rangle$, since $\cS_m$ does not involve $p_m$ (see \cite[p.\,2]{Stolz}). The first part of the lemma follows therefore from the triviality of the signature of $Q$. For $m=2k$, we use the formulas above to calculate
\[\cS_{2k}(Q)=\left(\frac{1}{2}(s_k^2-s_{2k})+\sigma_{2k}c_{2k}\frac{1}{2}(\hat{s}_k^2-\hat{s}_{2k})+\sigma_{2k}d_{2k}\hat{s}_k\frac{(-1)^{k+1}}{(2k-1)!}+\sigma_{2k}d_{2k}\frac{1}{2(4k-1)!}\right)p_{k}^2(Q).\] Using the second description of $s_{2k}$ in \eqref{equation:classcoefficients}, one obtains the identity
\[\frac{1}{2}s_{2k}=\sigma_{2k}d_{2k}\frac{1}{2(4k-1)!}-\sigma_{2k}c_{2k}\frac{1}{2}\hat{s}_{2k},\] which simplifies the above formula for $\cS_{2k}(Q)$ to
\[\cS_{2k}(Q)=\left(\frac{1}{2}s_k^2+\sigma_{2k}c_{2k}\frac{1}{2}\hat{s}^2_k+\sigma_{2k}d_{2k}\hat{s}_k\frac{(-1)^{k+1}}{(2k-1)!}\right)p_{k}^2(Q).\] Substituting $\hat{s}_k$ with its value and using \eqref{equation:pontryaginclass} as well as the last identity in \eqref{equation:classcoefficients}, we arrive at
\[\cS_{2k}(Q)=\left(\frac{\lambda_k^2\sigma_k^2}{j_k^2}+\lambda_k^2\sigma_{2k}\left(a_k^2c_{2k}(\frac{\lvert B_{2k}\rvert }{4k})^2+2(-1)^kd_{2k}\frac{\lvert B_{2k}\rvert }{4k}\right)\right)[Q,\partial Q]^*,\] from which we obtain the first formula of the statement, since $s(Q)=-1/8\langle \cS_{2k}(Q),[Q,\partial Q]\rangle$ as $\sigma(Q)=0$. The second formula follows from the first together with the identity
\[\sigma_{2k}c_{2k}=2^{4k+1}(2^{4k-1}-1)\left(1-\denom\left(\frac{\lvert B_{4k}\rvert }{8k}\right)d_{2k}\right)\] by combining two of the summands to $-a_k^2\lambda_k^2T_k^2/{16}$ and simplifying the expressions.
\end{proof}

To determine the combinations of Pontryagin numbers, signatures, and $\hat{A}$-genera that are realized by closed highly connected $4m$-manifolds, we recall that, even for almost closed $4m$-manifolds $M$, these invariants are additive bordism invariants, where the top-dimensional Pontryagin number $p_m(M)$ is defined such that the Hirzebruch signature formula $\sigma(M)=\langle\cL_m(M),[M]\rangle$ holds. Using this description of $p_m$, the $\hat{A}$-genus of an almost closed $4m$-manifold is defined by $\hat{A}(M)=\langle\hat{\cA}_m(M),[M]\rangle$.

For $m$ odd, \cref{proposition:kernel} shows that the torsion free quotient of $\Omega_{4m}^{\langle 2m-1\rangle}$ is generated by the class $(\sigma_m/8)\cdot P$, whose invariants can be easily computed from $\sigma(P)=8$ and $p_k^2(P)=0$ using the formulas recalled at the beginning of this subsection.

\begin{prop}\label{proposition:latticeodd}For $m\neq1$ odd, the torsion free quotient of $\Omega^{\langle 2m-1\rangle}_{4m}$ is generated by $(\sigma_m/8)\cdot P$, whose signature, $\hat{A}$-genus and Pontryagin numbers are
\[\sigma\big((\sigma_m/8)\cdot P\big)=\sigma_m,\quad\hat{A}\big((\sigma_m/8)\cdot P\big)=-2\num\left(\frac{\lvert B_{2m}\rvert }{4m}\right),\quad\text{and}\]\[p_{m}\big((\sigma_m/8)\cdot P\big)=2(2m-1)!j_m.\] 
\end{prop}

For $m$ even, we compute the occurring characteristic numbers as follows.

\begin{prop}\label{proposition:latticeeven}The torsion free quotient of $\Omega^{\langle 4k-1\rangle}_{8k}$ agrees with $\Omega_{8k}^{\langle 4k-1\rangle}/\Theta_{8k}$, and the morphisms induced by the respective characteristic numbers
 \[\sigma,\hat{A}, p_{2k}, p_k^2\colon\Omega^{\langle 4k-1\rangle}_{8k}/\Theta_{8k}\ra\bfZ\] are, with respect to the ordered basis described in \cref{proposition:kernel}, given by
\[\sigma=\begin{pmatrix}
    \sigma_{2k}\\
    \ord([\Sigma_Q])\frac{a_k^2}{\mu_k}\left(\frac{{T_k}^2}{2}-2\sigma_{2k}d_{2k}\frac{\lvert B_{2k}\rvert }{4k}\left(\frac{\lvert B_{2k}\rvert }{\lvert B_{4k}\rvert }+(-1)^{k+1}\right)\right)
\end{pmatrix},\]
\[\hat{A}=\begin{pmatrix}
    -\num\left(\frac{\lvert B_{4k}\rvert }{8k}\right)\\
    2\ord([\Sigma_Q])\frac{a_k^2}{\mu_k}\num\left(\frac{\lvert B_{4k}\rvert }{8k}\right)d_{2k}\frac{\lvert B_{2k}\rvert }{4k}\left(\frac{\lvert B_{2k}\rvert }{\lvert B_{4k}\rvert }+(-1)^{k+1}\right)
\end{pmatrix},\]
\[{p_{2k}}=\begin{pmatrix}
    (4k-1)!j_{2k}\\
    \ord([\Sigma_Q])\frac{a_k^2}{\mu_k}\left((2k-1)!^2+(4k-1)!j_{2k}\frac{\lvert B_{2k}\rvert }{4k}\left(c_{2k}\frac{\lvert B_{2k}\rvert }{4k}+2d_{2k}(-1)^k\right)\right)\\
\end{pmatrix},\text{ and}\]
\[{p_k^2}=\begin{pmatrix}
    0\\
    2\ord([\Sigma_Q])\frac{a_k^2}{\mu_k}(2k-1)!^2
\end{pmatrix},\] where $\mu_k=2$ if $k=1,2$ and $\mu_k=1$ otherwise. Furthermore, after replacing $\mu_k$ with its reciprocal, the same formulas hold for the subgroup $\Omega^{\langle 4k-1\rangle}_{8k}/\Theta_{8k}\cap\sigma^{-1}(4\cdot\bfZ)$ using the ordered basis described in \cref{remark:usefulbasis}.
\end{prop}

\begin{proof}The first statement follows from \cref{proposition:kernel} and the computation of the invariants for $(\sigma_m/8)\cdot P$ is analogous to \cref{proposition:latticeodd}, using the fact that $p_k^2(P)$ vanishes since $P$ is parallelizable.
The second part of \cref{corollary:s(Q)consequences} and \cref{lemma:alternativebasis} imply that $\ord([\Sigma_Q])\big(Q-s(Q)\cdot P\big)$ equals $8\cdot \bfH P^2$ for $k=1$ and $8\cdot \bfO P^2$ for $k=2$. Since $\mu_k\lambda_k^2$ is $8$ if $k=1,2$ and $1$ otherwise, it is sufficient to show that the invariants of the class $\ord([\Sigma_Q])\big(Q-s(Q)\cdot P\big)$ are in all cases given by the product of $\mu_k\lambda_k^2$ with the claimed invariants of the second basis vector. For $p_k^2$, this is implied by $p_k^2(P)=0$ and \eqref{equation:pontryaginclass}. The signature of $\ord([\Sigma_Q])\big(Q-s(Q)\cdot P\big)$ is obtained using the second formula in \cref{lemma:s(Q)}, together with $\sigma(P)=8$ and $\sigma(Q)=0$. Using \eqref{equation:Ahat}, the values of these signatures, together with $p_k^2(P)=0$ and \eqref{equation:pontryaginclass}, result in
\[\hat{A}\left(\ord([\Sigma_Q])\left(Q-s(Q)\cdot P\right)\right)=\ord([\Sigma_Q])\left(\frac{\lambda_k^2a_k^2T_k^2}{2^{4k+2}(2^{4k-1}-1)}+\frac{s(Q)}{2^{4k-2}(2^{4k-1}-1)}\right),\] which, using the second formula of \cref{lemma:s(Q)}, gives the desired value. The calculation of $p_{2m}$ is obtained by combining \eqref{equation:p2k} with the first formula for $s(Q)$ of \cref{lemma:s(Q)} and 
\[\frac{\lambda_k^2s_k^2}{s_{2k}}a_k^2(2k-1)!^2=\frac{\lambda_k^2\sigma_k^2}{s_{2k}j_k^2}.\] The last part of the statement follows from \cref{remark:usefulbasis}.\end{proof}

We proceed by computing the signatures and $\hat{A}$-genera realized by highly connected $4m$-manifolds. The following consequences of the von Staudt--Clausen theorem will be useful. A proof can be derived, for instance, from \cite[Ch.\,3]{BernoulliZeta}.

\begin{thm}[von Staudt--Clausen]\label{theorem:vSC}The prime factor decomposition of $\denom(\frac{\lvert B_{2n}\rvert }{n})$ is
\[\denom\left(\frac{\lvert B_{2n}\rvert }{n}\right)=\prod_{p-1|2n} p^{1+\nu_p(n)}.\] In particular, $j_{2n}=\denom(\frac{\lvert B_{4n}\rvert }{8n})$ is divisible by $j_n=\denom(\frac{\lvert B_{2n}\rvert }{4n})$, and $\nu_2(j_n)=\nu_2(n)+3$.
\end{thm}

\begin{prop}\label{proposition:minimalsignature} 
There exists a closed $(2m-1)$-connected $4m$-manifold with signature \[\begin{cases}\sigma_{m}&\mbox{if }m\neq1\text{ is odd}\\ 2^{i_m}\gcd (\sigma_{m},\sigma_{m/2}^{2}) &\mbox{if }m\neq2,4\text{ is even}\\1 &\mbox{if }m=1,2,4,\end{cases}\] where \[i_m=\min\left(0,\nu_2(\ord([\Sigma_Q]))-2\nu_2(m)-4+2\nu_2(a_{m/2})\right),\] and the signature of any closed $(2m-1)$-connected $4m$-manifold is a multiple of this number.
\end{prop} 

\begin{rem}
Note that we obtain $\nu_2(\ord([\Sigma_Q]))\le2\nu_2(m)+4$ from \cref{theorem:AndersonSchultz} and \ref{theorem:vSC}.
\end{rem}

\begin{rem}\label{remark:comparisonalmostparallelizability}\cref{proposition:minimalsignature} should be compared with the analogous result for $(2m-1)$-connected $4m$-manifolds that are assumed to be almost parallelizable. It follows from work of Milnor--Kervaire \cite[p.\,457]{KervaireMilnorBernoulli} that, under this additional assumption, the minimal positive signature that occurs is $\sigma_m$, independently  of the parity of $m$. Since $i_m\le0$, one gets \[2^{i_m}\gcd (\sigma_{m},\sigma_{m/2}^{2})<\frac{\sigma_m}{2^{m-\nu_2(m)-8}}\] by \cite[Thm 1.5.2(c)]{AntonelliBurgheleaKahn}, so the minimal positive signature becomes, for $m$ even, significantly larger if one restricts to manifolds that are almost parallelizable. 
\end{rem}

The following proof of \cref{proposition:minimalsignature} owes a significant intellectual debt to a computation due to Lampe \cite[Satz 1.3]{Lampe}, to whom the authors would like to express their gratitude for sending them a copy of his diploma thesis.

\begin{proof}[Proof of \cref{proposition:minimalsignature}]For $m\neq 1$ odd, the claim is a consequence of \cref{proposition:latticeodd}, and for $m=1,2,4$, it follows from the existence of the manifolds $\bfC P^2$, $\bfH P^2$, and $\bfO P^2$ of signature $1$. For $m=2k\neq2,4$, it is, by \cref{proposition:latticeeven}, sufficient to show that the integer described in the statement is the greatest common divisor of
\[\sigma\left((\sigma_{2k}/8)\cdot P\right)=\sigma_{2k}\quad\text{and}\quad\sigma\left(\ord([\Sigma_Q])\big(Q-s(Q)\cdot P\big)\right)=-8\ord([\Sigma_Q])s(Q).\] From the first formula of \cref{lemma:s(Q)}, one sees that there is an odd integer $b$ such that
\[-8\ord([\Sigma_Q])s(Q)=\frac{1}{j_k^2}\left(\ord([\Sigma_Q])(\sigma_k^2+\sigma_{2k}a_k^2b)\right).\] Using $\nu_2(\sigma_k)=\nu_2(a_k)+2k+1$ and $\nu_2(j_k)=\nu_2(k)+3$, we compute 
\[\nu_2\left(\gcd\left(\sigma_{2k},8\ord([\Sigma_Q])s(Q)\right)\right)=\min\left(4k+1,\nu_2\left(\ord([\Sigma_Q])\right)+4k+2\nu_2(a_k)-2\nu_2(k)-5\right).\] 
To determine the odd part $\gcd(\sigma_{2k},8\ord([\Sigma_Q])s(Q))_{\odd}$ of the greatest common divisor, we note that, since $T_k$ is an even integer for $k>1$, the number $j_k$ divides $2^{2k-1}(2^{2k}-1)$. This, together with the fact that $(2^{2k}-1)$ and $(2^{4k-1}-1)$ are coprime, implies that $j_k$ and $(2^{4k-1}-1)$ are coprime. But since $j_k$ is also coprime to $\num(\frac{\lvert B_{4k}\rvert }{8k})$ by the von Staudten--Clausen theorem, it cannot share odd prime divisors with $\sigma_{2m}$. As $\ord([\Sigma_Q])$ divides $j_k^2$ by \cref{theorem:AndersonSchultz}, the conclusion also holds for $\ord([\Sigma_Q])$ and $\sigma_{2m}$. This leads to
\[\gcd\left(\sigma_{2k},8\ord([\Sigma_Q])s(Q)\right)_{\odd}=\gcd\left(\sigma_{2k},\ord([\Sigma_Q])(\sigma_k^2+\sigma_{2k}a_k^2b)\right)_{\odd}=\gcd(\sigma_{2k},\sigma_k^2)_{\odd},\] which implies the statement, because $\nu_2(\gcd(\sigma_{2k},\sigma_k^2))=4k+1$.
\end{proof} 

Since $\nu_2(\sigma_m)=2m+1+\nu_2(a_m)$ and $i_m\ge -2\nu_2(m)-4$, we obtain the following divisibility result for the signature as corollary of \cref{proposition:minimalsignature}. 

\begin{cor}\label{corollary:divisibilitysignature}The signature of a closed $(2m-1)$-connected manifold of dimension $4m$ for $m\neq1,2,4$ is divisible by $2^{2m+2}$ if $m$ is odd and by $2^{2m-2\nu_2(m)-3}$ if $m$ is even.
\end{cor}

For $m\ge 2$, closed $(2m-1)$-connected $4m$-manifolds admit a spin structure, so their $\hat{A}$-genus is integral. We compute it as follows.

\begin{prop}\label{proposition:minimalAhat}
There exists a closed $(2m-1)$-connected $4m$-manifold with $\hat{A}$-genus \[\begin{cases}2\num(\frac{\lvert B_{2m}\rvert }{4m})&\mbox{if }m\neq1\text{ is odd}\\\gcd(\num(\frac{\lvert B_{2m}\rvert }{4m}),\num(\frac{\lvert B_{m}\rvert }{2m})^2)&\mbox{if }m\text{ is even,}\end{cases}\] and the $\hat{A}$-genus of any closed $(2m-1)$-connected $4m$-manifold is a multiple of this number.
\end{prop}

\begin{proof}Arguing similarly as in the proof of \cref{proposition:minimalsignature}, the claim for $m$ odd follows from \cref{proposition:latticeodd}. To prove the case of $m$ even using \cref{proposition:latticeeven}, we need to compute the greatest common divisor of $\num(\frac{\lvert B_{4k}\rvert }{8k})$ and 
\begin{multline*}2\ord([\Sigma_Q])\frac{a_k^2}{\mu_k}\num\left(\frac{\lvert B_{4k}\rvert }{8k}\right)d_{2k}\frac{\lvert B_{2k}\rvert }{4k}\left(\frac{\lvert B_{2k}\rvert }{\lvert B_{4k}\rvert }+(-1)^{k+1}\right)\\=\frac{1}{j_k^2\mu_k}\ord([\Sigma_Q])a_k^2d_{2k}\num\left(\frac{\lvert B_{2k}\rvert }{4k}\right)\left(\num\left(\frac{\lvert B_{2k}\rvert }{4k}\right)j_{2k}+2(-1)^{k+1}j_k\num\left(\frac{\lvert B_{4k}\rvert }{8k}\right)\right)\end{multline*} As $j_k, j_{2k}, a_k,\mu_k$, and $d_{2k}$ are coprime to $\num(\frac{\lvert B_{4k}\rvert }{8k})$, the number in question agrees with \[\gcd\left(\num\left(\frac{\lvert B_{4k}\rvert }{8k}\right),\ord([\Sigma_Q])\num\left(\frac{\lvert B_{2k}\rvert }{4k}\right)^2\right).\] But $\ord([\Sigma_Q])$ divides $j_k^2$ by \cref{theorem:AndersonSchultz}, so the von Staudt--Clausen theorem implies that it has no common divisors with $\num(\frac{\lvert B_{4k}\rvert }{8k})$. This yields the result.
\end{proof}

\begin{rem}\label{remark:computercalculationsignature}\ 
\begin{enumerate}
\item \label{remark:computercalculationAhat}Computer calculations show that for $m<2678$ even, the greatest common divisor of $\sigma_m$ and ${\sigma_{m/2}}^2$ is a power of $2$, whereas it is  $2^{2\cdot2678+1}\cdot 34511$ in the case $m=2678$. Since $\nu_2(\gcd(\sigma_m,{\sigma_{m/2}}^2))=2m+1$, the minimal positive signature of a closed highly connected $4m$-manifold for $m\neq 2,4$ even, $m< 2678$, equals $2^{i_m+2m+1}$ by \cref{proposition:minimalsignature}. 

\item Similar computations show that $\num(\frac{\lvert B_{2m}\rvert }{4m})$ and $\num(\frac{\lvert B_{m}\rvert }{2m})^2$ are coprime for $m<44000$ even, and hence \cref{proposition:minimalAhat} implies that in these dimensions, there exists a closed $(2m-1)$-connected $4m$-manifold of $\hat{A}$-genus $1$. We do not know whether $\num(\frac{\lvert B_{2m}\rvert }{4m})$ and $\num(\frac{\lvert B_{m}\rvert }{2m})^2$ are coprime for all even $m$.  
\end{enumerate}
\end{rem}

%%%%%%%%%
%%%%%%%%%

\section{Characteristic numbers of bundles over surfaces with highly connected fiber}\label{section:lattices}
 This section serves to derive consequences for characteristic numbers of smooth bundles over surfaces and for $\oH^2(\BDiff(M);\bfZ)$, leading to proofs of \cref{bigthm:signatures_Ahat}, \cref{bigcor:signatures}, and \cref{bigtheorem:cobordismlattice}. Throughout this section, we require $M$ to be a smooth, closed, $2n$-dimensional, highly connected, almost parallelizable (or, equivalently, $n$-parallelizable) manifold.

\subsection{Characteristic numbers of total spaces}
In \cref{proposition:latticeodd} and \ref{proposition:latticeeven}, we determined the lattices of Pontryagin numbers, signatures and $\hat{A}$-genera realized by classes in the subgroup of oriented bordism classes
\[\im\left(\Omega_{2n+2}^{\langle n \rangle}\ra\Bord_{2n+2}\right) \cap \sigma^{-1} (4 \cdot \bfZ)\subseteq\Bord_{2n+2},\]
which, combined with \cref{bigtheorem:bordismimage}, provides divisibility constraints for the respective invariants for total spaces of smooth $M$-bundles over surfaces and completely computes these invariants as long as $g(M)\ge5$. The consequences for the signature and $\hat{A}$-genus of our computation, which are carried out in \cref{proposition:minimalsignature}, \cref{corollary:divisibilitysignature}, and \cref{proposition:minimalAhat} imply
\cref{bigthm:signatures_Ahat} as well as \cref{bigcor:signatures}.

\subsection{Generalized Miller--Morita--Mumford classes}\label{section:MMMclassessmoothly}We prove \cref{bigtheorem:cobordismlattice}, i.e., the computation of the torsion free quotient $\oH^2(\BDiff(M);\bfZ)_{\free}$ for $g(M)\ge7$ and $2n\ge6$ in terms of generalized Miller--Morita--Mumford classes (as recalled in the introduction). The root of the computation is the analogous result for rational cohomology is due to Galatius--Randal-Williams: their high dimensional analogue of the Madsen--Weiss theorem (see \cref{theorem:parametrizedpontryaginthom}) provides an isomorphism \[\oH^*\left(\BDiffuo(M,D^{2n});\bfZ\right) \cong \oH^*(\Omega^\infty_M\MT\theta_n;\bfZ)\]  for $2n\ge6$, $g(M)\ge7$, and $*\le 2$. The cohomology groups $\oH^*(\Omega^\infty_M\MT\theta_n;\bfZ)$ are finitely generated and can be computed rationally to be the free graded commutative algebra
\[\oH^*(\Omega^\infty_M\MT\theta_n;\bfQ) \cong \Lambda\left(\oH^{*+2n>0}(\BSO(2n)\langle n\rangle;\bfQ)\right),\] see \,\cite[Ch.\,2.5]{GRWstable}. 
As $\oH^{*}(\BSO(2n)\langle n\rangle;\bfQ)$ is a polynomial ring in the Pontryagin classes $p_i$ of degree $4i>n$, we arrive at
\begin{equation}\label{equ:computation}\oH^2\left(\BDiffuo(M,D^{2n});\mathbf{Q}\right) = \begin{cases}
0  & \text{if } 2n \equiv 0 \Mod{4}\\
\mathbf{Q}\kappa_{p_{(n+1)/2}} & \text{if } 2n \equiv 2 \Mod{8}\\
\mathbf{Q}\kappa_{p_{(n+1)/2}} \oplus \mathbf{Q}\kappa_{p_{(n+1)/4}^2} & \text{if } 2n \equiv 6 \Mod{8}
\end{cases}
\end{equation} for $2n\ge6$ and $g\ge7$. Evaluating diffeomorphisms of $M$ at a fixed point $M$ induces the first homotopy fiber sequence of 
\[\BDiffuo(M,*)\ra \BDiff(M)\ra M\quad \text{and}\quad \BDiffuo(M,D^{2n})\ra \BDiff(M,*)\ra \BSO(2n),\] whereas the second one is given by taking the derivative at $*\in M$. From the Serre spectral sequence of these fibrations, one sees that the formula \eqref{equ:computation}, as well as the finite generation property, holds equally well for $\BDiff(M)$ instead of $\BDiffuo(M,D^{2n})$, an observation which incidentally yields the description of $\oH^2(\BDiff(M);\bfQ)$ mentioned in the introduction. Excluding the trivial case, we assume $4m=2n+2$. Since classes in $\Omega_{4m}^{\langle 2m-1\rangle}$ have vanishing Stiefel--Whitney numbers and are, up to torsion, detected by Pontryagin numbers (see e.g.~\cref{proposition:latticeodd} and \ref{proposition:latticeeven}), the induced morphism $(\Omega_{4m}^{\langle 2m-1\rangle})_{\free}\ra\Bord_{4m}$, defined on the torsion free quotient, is injective, and the rank of its image agrees with that of $\Bord_2(\BDiffuo(M,D^{2n}))_{\free}$. From this, using 
\cref{bigtheorem:bordismimage}, we conclude that the morphism $\Bord_2(\BDiff(M))\ra\Bord_{4m}$ induces an isomorphism of the form
\begin{equation}\label{eq:BordismGroupIdentification}
\Bord_2\left(\BDiff(M)\right)_{\free}\cong\im\left(\Omega^{\langle 2m-1\rangle}_{4m}\ra\Bord_{4m} \right)\cap \sigma^{-1} (4 \cdot \bfZ).
\end{equation}
 In view of the commutative diagram \eqref{eq:CommutativeDiagramIntroduction} of the introduction, the functional on the left hand side of \eqref{eq:BordismGroupIdentification} induced by a class $\kappa_{c}$ for $c\in\oH^{4m}(\BSO;\bfZ)$ corresponds via this isomorphism to the usual characteristic number defined by $c$ on the right hand side. Propositions \ref{proposition:latticeodd} and \ref{proposition:latticeeven} provide an explicit basis for the right hand side and represent the functionals induced by Pontryagin numbers in terms of this basis. Computing the appropriate change of basis matrix, the integral dual of this basis can be expressed in terms of Pontryagin numbers, which results in the basis of $\oH^2(\BDiff(M);\bfZ)_{\free}$ described in \cref{bigtheorem:cobordismlattice}. 
 
 \begin{rem}The conclusions of \cref{bigtheorem:cobordismlattice} are also valid for $\BDiffuo(M,D^{4m-2})$ instead of $\BDiff(M)$, by the same argument.
 \end{rem}

%%%%%%%%%
%%%%%%%%%

\section{Topological bundles}\label{TopologicalBundles}
The purpose of this final section is to study the topological analogue of the bundles considered in the previous sections, i.e., fiber bundles $\pi\colon E\ra S$ over oriented surfaces with fiber a smooth highly connected almost parallelizable manifold $M$ of even dimension and structure group $\Homeo(M)$ as opposed to $\Diff(M)$. The main work in this section is to compute the values of the Pontryagin numbers of the total spaces of such bundles, or equivalently, the values of the Miller--Morita--Mumford classes of $\pi$. We shall see that the bordism group of topological $M$-bundles agrees rationally with the analogous group for smooth bundles, but differs integrally. This gives rise to obstructions to smoothing topological $M$-bundles over surfaces up to bordism.

\subsection{Rational invariance}
We begin by showing that the bordism group of $M$-bundles over surfaces is rationally independent of whether we consider smooth or topological bundles (at least in the presence of large genus), or equivalently, that the canonical map between second cohomology groups $\oH^2(\BHomeo(M);\bfQ)\ra \oH^2(\BDiff(M);\bfQ)$ is an isomorphism.

\begin{prop}\label{lem:homeodiffrationally}Let $M$ be a closed, smooth, $(n-1)$-connected, almost parallelizable $2n$-manifold. If $g(M)\ge7$ and $2n\ge6$, then the canonical map
\[\oH^2(\BHomeo(M);\bfQ)\ra \oH^2(\BDiff(M);\bfQ)=\begin{cases}
0  & \text{if } 2n \equiv 0 \Mod{4}\\
\mathbf{Q}\kappa_{p_{(n+1)/2}} & \text{if } 2n \equiv 2 \Mod{8}\\
\mathbf{Q}\kappa_{p_{(n+1)/2}} \oplus \mathbf{Q}\kappa_{p_{(n+1)/4}^2} & \text{if } 2n \equiv 6 \Mod{8}
\end{cases}\] is an isomorphism.
\end{prop}

\begin{proof}
The computation of $\oH^2(\BDiff(M);\bfQ)$ follows from \eqref{equ:computation} and the subsequent discussion, and shows that this group is generated by generalized Miller--Morita--Mumford classes associated to monomials in Pontryagin classes, so the surjectivity of the first morphism follows from the topological invariance of rational Pontryagin classes. For the injectivity part, we first show that the homotopy fiber $\Homeo(M)/\Diff(M)$ of the canonical map $\BDiff(M)\ra \BHomeo(M)$ has finitely many components and vanishing first rational cohomology. To see this, recall that smoothing theory (see e.g.\,\cite[Thm 4.2]{BurgheleaLashofSmoothing}) identifies $\Homeo(M)/\Diff(M)$ with a collection of path-components of the space $\Gamma(TM)$ of lifts
\begin{center}
\begin{tikzcd}
&\BSO(2n)\arrow[d]\\
M\arrow[r,swap]\arrow[ur,dashed,bend left]&\BSTop(2n)
\end{tikzcd}
\end{center} of a classifying map for the oriented topological tangent bundle of $M$, at least if we use a model for the canonical map $\BSO(2n)\ra\BSTop(2n)$ that is a fibration. Restricting such lifts to the $n$-skeleton $\vee^{k}S^n\subset M$ induces a map $\Gamma(TM)\ra\Maps(\vee^{k}S^n,\STop(2n)/\SO(2n))$ whose homotopy fibers are either empty or equivalent to $\Maps(S^{2n},\STop(2n)/\SO(2n))$. From this, we see that the claim on $\Homeo(M)/\Diff(M)$ follows from the fact that  $\pi_k\STop(2n)/\SO(2n)$ is finite for $*\le 2n+1$ (see e.g.\,\cite[Cor.\,5.2 b), p.\,38 (7)]{BurgheleaLashofSmoothing}). Since $\oH^1(\Homeo(M)/\Diff(M);\bfQ)$ vanishes, the edge homomorphism \[\oH^2(\BHomeo(M);\oH^0(\Homeo(M)/\Diff(M);\bfQ))\ra \oH^2(\BDiff(M);\bfQ)\] of the corresponding Serre spectral sequence is injective. The morphism in the claim agrees with this edge homomorphism, precomposed with the map \[\oH^2(\BHomeo(M);\bfQ)\ra\oH^2(\BHomeo(M);\oH^0(\Homeo(M)/\Diff(M);\bfQ))\] induced by the map of coefficient systems $\varepsilon\colon \bfQ\ra \oH^0(\Homeo(M)/\Diff(M);\bfQ)$ resulting from $\Homeo(M)/\Diff(M)\ra *$. Consequently, the claim follows from $\varepsilon$ being split injective, which holds by transfer since  $\oH^0(\Homeo(M)/\Diff(M);\bfQ)$ is a finite permutation module as $\Homeo(M)/\Diff(M)$ has finitely many components.
\end{proof}

\subsection{Topological Pontryagin numbers}
As a next step, we determine the lattice spanned by the values of the characteristic classes $\kappa_{p_{(n+1)/2}}$ and $\kappa_{p^2_{(n+1)/4}}$ on topological $M$-bundles over surfaces, or equivalently the values of the corresponding Pontryagin classes on the total spaces $E$ of such bundles. It turns out to be more convenient to consider the different basis $p^2_{(n+1)/4}$ and $\cL_{(n+1)/2}$ first, so we shall do that. In this subsection, we compute the minimal value of $p^2_{(n+1)/4}(E)$ and show that it can be realized by a total space $E$ with $\cL_{(n+1)/2}(E)=0$. Given this, to determine the lattice, it suffices to compute the values of $\cL_{(n+1)/2}(E)$, which we shall do in the next subsection by proving \cref{bigthm:signaturestopologically}.

We begin with a standard lemma on highly connected topological manifolds.

\begin{lem}\label{lem:topologicallyevenform}Let $E$ be an oriented, closed, highly connected, topological manifold of dimension $4m\neq4,8,16$.
\begin{enumerate}
\item The intersection form of $E$ is even. In particular, the signature $\sigma(E)$ is divisible by $8$.
\item For $m=2k$, the Pontryagin number $p_k(E)^2$ is contained in the subgroup $2(p^{\min}_k)^2\cdot\bfZ\subseteq\bfQ$, where $p^{\min}_k\in\bfQ$ is an additive generator of the image of $p_k\colon\pi_{4k}\BSTop\ra\bfQ$.
\end{enumerate}
\end{lem}
\begin{proof}As $E$ is $(4k-1)$-connected, we can represent the middle homology \[\textstyle{\pi_{4k}(E)\cong\oH_{4k}(E;\bfZ)\cong\bigoplus^l_{i=1}\oH_{4k}(S^{4k}_i)}\] by locally flatly embedded $4k$-spheres $S^{4k}_i\subseteq E$ (see e.g.\,\cite[Thm 1]{Pedersen}), which admit topological normal bundles $\nu_i\colon S^{4k}_{i}\ra\BSTop(4k)$ by \cite[Cor.\,5.5]{RourkeSanderson}. Writing $(a_{ij})_{1\le i,j\le l}$ for the (symmetric) intersection matrix of $E$ with respect to the basis given by the $S^{4k}_i$, claim (i) follows from showing that the diagonal entries $a_{ii}$ are even. This can be seen, as in the smooth case, by using that $a_{ii}$ agrees with the image of $\nu_i$ under the composition
\[\pi_{4k}\STop(4k)\ra\pi_{4k}\BSG(4k)\cong \pi_{4k-1}S^{4k}\xrightarrow{H}\bfZ,\] where the first morphism assigns to a topological $\bfR^{4k}$-bundle its underlying spherical fibration and $H$ is the Hopf invariant, whose values are even as long as $4k\neq4,8$. This also implies (ii), since we can compute the Pontryagin number as
\[\textstyle{p_k^2(E)=\sum_{1\le i<j\le m}2a_{ij}p^2_k(\nu_i)+\sum_{1\le i\le m}a_{ii}p^2_k(\nu_i)},\] which is contained in $2(p^{\min}_k)^2\cdot\bfZ\subset\bfQ$ as $a_{ii}$ is even and $p_k(\nu_i)\in p_k^{\min}\cdot\bfZ\subseteq\bfQ$ .
\end{proof}

The value $p_k^{\min}$ appearing in the previous lemma was determined by Brumfiel as follows.

\begin{lem}[Brumfiel] \label{lem:BrumfielMinimalPontryagin} For $k\neq1,2$, the image of $p_k\colon \pi_{4k}\BSTop\ra \bfQ$ is generated by \[p_k^{\min}= \frac{8}{\sigma_k}a_k(2k-1)!2^{\nu_2(k)+3}.\]
\end{lem}
\begin{proof}The proof of \cite[Lem.\,4.5]{Brumfiel} shows that, after precomposition with the natural map $\pi_{4k}\BSPL\ra\pi_{4k}\BSTop$, the image of the map in question is generated by $(1/\lvert \bP_{4k}\rvert )a_k(2k-1)!2^{\nu_2(j_k)}$. But $\pi_{4k}\BSPL\ra\pi_{4k}\BSTop$ is an isomorphism for $k\neq 1$ as $\STop/\SPL\simeq K(\bfZ/2,3)$ by \cite[Essay V Thm 5.5]{KirbySiebenmann}, so the claim follows from the identities $\lvert\bP_{4k}\rvert=\sigma_k/8$ and $\nu_2(j_k)=\nu_2(k)+3$ (see \cref{theorem:vSC}).
\end{proof}

Before turning towards the values of $p_k^2(E)$, we prove two more auxiliary lemmas.

\begin{lem}\label{lem:surjectivityofdiscbundles}Restricting homeomorphisms of $D^{4k-1}$ fixing the center to the interior $\interior{D^{4k-1}}\cong\bfR^{4k-1}$ and stabilising by $(-)\times \id_\bfR$ induces a composition
\[\pi_{4k}\BHomeo(D^{4k-1},0)\ra \pi_{4k}\BSTop(4k-1)\ra\pi_{4k}\BSTop\] whose first morphism is surjective for all $k\ge1$ and the second one for $k\neq1,2$.
\end{lem}
\begin{proof}
Both claims follow from the fact that the relative homotopy groups $\pi_*(\STop/\SO,\STop(4k-1)/\SO(4k-1))$ vanish for $*\le 4k+1$ (see e.g.\,\cite[Essay V,§5,5.0]{KirbySiebenmann}). This is explained in \cite[Cor.\,1.2 (i)]{Stern} for the first map, and follows for the second one from a diagram chase in the commutative diagram with exact rows
\begin{center}
\begin{tikzcd}[column sep=0.4cm, row sep=0.5cm]
\pi_{4k}\BSO(4k-1)\rar\arrow[d,two heads]&\pi_{4k}\BSTop(4k-1)\dar\arrow[r,two heads]&\pi_{4k-1}\STop(4k-1)/\SO(4k-1)\arrow[d,"\cong"]\\
\pi_{4k}\BSO\rar&\pi_{4k}\BSTop\rar&\pi_{4k-1}\STop/\SO,
\end{tikzcd}
\end{center}
which is induced by the appropriate map of fiber sequences. The right vertical map is an isomorphism because of the above mentioned vanishing of $\pi_*(\STop/\SO,\STop(4k-1)/\SO(4k-1))$ in a range, the surjectivity of the left vertical map for $k\neq1,2$ is standard (c.f.\,\cite[§1 B)]{Levine}), and the upper right horizontal map is surjective as $\pi_{4k-1}\BSO(4k-1)\ra\pi_{4k-1}\BSTop(4k-1)$ is injective (see e.g.\,\cite[Prop.\,5.4 (iv)]{BurgheleaLashofSmoothing}).
\end{proof}

\begin{lem}\label{lem:ConstructionBundleMinimalPontryagin}
Given classes $\xi,\eta \in\pi_{4k}\BSTop(4k-1)$ and a closed, highly connected, almost parallelizable $(8k-2)$-manifold $M$ with $g(M)\ge1$, there exists an oriented topological $M$-bundle $\pi\colon E \ra S^1\times S^1$ such that \[p_k^2(E)=2 p_k(\xi)p_k(\eta)\quad\text{and}\quad\sigma(E) = 0.\]
\end{lem}
\begin{proof}We first prove the claim for $M=W_1\coloneq S^{4k-1}\times S^{4k-1}$. \cref{lem:surjectivityofdiscbundles} ensures that we can pick representatives $\xi,\eta \colon S^{4k}\ra \Homeo(D^{4k-1},0)$ of the respective elements in $\pi_{4k} \BSTop(4k-1)$, and we shall choose them such that they are constantly the identity when restricted to a neighborhood of the lower hemisphere $D^{4k}_-\subseteq S^{4k}$. The map $\xi$ gives rise to a homeomorphism $f_{\xi}$ of $S^{4k-1} \times D^{4k-1}_-$ by mapping $(x,y)$ to $(x,\xi(x)(y))$, which is the identity on a neighborhood of $D^{4k-1}_-\times D^{4k-1}_-$. Similarly, by switching the role of the two coordinates and replacing $\xi$ by $\eta$, we obtain a homeomorphism $f_{\eta}$ of $D^{4k-1}_- \times S^{4k-1}$. Extending $f_{\xi}$ and $f_{\eta}$ the identity gives rise to two strictly commuting homeomorphism $f'_{\xi}$ and $f'_\eta$ of
\[W_{1,1}\coloneq S^{4k-1}\times S^{4k-1}\backslash\interior{D^{8k-2}} \cong  S^{4k-1} \times D_-^{4k-1}  \cup_{D_-^{4k-1} \times D_-^{4k-1}} D_-^{4k-1} \times S^{4k-1}.\]  Their mapping tori combine to a $W_{1,1}$-bundle over $S^1\vee S^1$ which we can arrange to be trivialized over a neighborhood of the base point. As  $f'_{\xi}$ and $f'_\eta$ commute strictly, this bundle has a canonical extension over the $2$-cell of $S^1\times S^1$, so induces a $W_{1,1}$-bundle $\widetilde{E} \to S^1 \times S^1$, from which we obtain a $W_1$-bundle $E \to S^1 \times S^1$ by fiberwise coning off the boundary of $W_{1,1}$. To show that this bundle has the claimed properties, note that $f'_{\xi}$ and $f'_\eta$ fix the deformation retract $S^{4k-1} \vee S^{4k-1} = S^{4k-1} \times \{0\} \cup_{\{(0,0)\} } \{0\} \times S^{4k-1} \subset W_{1,1}$ pointwise, so the $W_{1,1}$-bundle $\widetilde{E} \to S^1 \times S^1$ is trivial as a fibration and $\widetilde{E}$ contains $\widetilde{E}^{\operatorname{core}} = (S^{4k-1} \vee S^{4k-1}) \times (S^1 \times S^1)$ canonically as a deformation retract. As $\widetilde{E}^{\operatorname{core}} \subset E$ induces an isomorphism on $\oH^{4k}(-)$, to determine the Pontryagin classes of the stable topological tangent bundle of $E$, it suffices to do so when pulled back to $\widetilde{E}^{\operatorname{core}}$. Since $f_{\xi}'$ and $f_{\eta}'$ are the identity on $D^{4k-1}_-\times S^{4k-1}\subset W_{1,1}$ and $S^{4k-1}\times D^{4k-1}_-\subset W_{1,1}$ respectively, the submanifolds $(* \vee S^{4k-1}) \times (S^1 \vee *)$ and $(S^{4k-1}\vee *) \times (*\vee S^1)$ come with a trivialized disk normal bundle in $E$, so the pullback of stable tangent bundle of $E$ is trivial over them. Furthermore, by construction, the compositions
\[(S^{4k-1} \vee *) \times (S^1 \vee *)\ra S^{4k}\xra{\xi\oplus \varepsilon} \BSTop(4k) \quad\text{and}\quad (* \vee S^{4k-1}) \times (* \vee S^1)\ra S^{4k}\xra{\eta\oplus \varepsilon} \BSTop(4k)\] classify normal bundles of the other two summands of $\widetilde{E}^{\operatorname{core}}$, where the first maps are given by collapsing the $(4k-1)$-skeleton and the second ones by stabilising $\xi$ and $\eta$, so we get \[p_k(E)=p_k(\xi)\cdot[(S^{4k-1} \vee *) \times (S^1 \vee *)]^{*} +p_k(\eta)\cdot [* \vee (S^{4k-1}) \times (* \vee S^1)]^{*} \in\oH^{4k}(E;\bfQ),\] where $(-)^*$ denotes Poincaré dual. Since $(S^{4k-1} \vee *) \times (S^1 \vee *)$ and $(* \vee S^{4k-1}) \times (* \vee S^1)$ intersect in $E$ transversely in a point and have normal bundles with trivial Euler class as they destabilize to $(4k-1)$-bundles, we conclude $p_k(E)^2=2p_k(\xi)p_k(\eta)\cdot1^*\in\oH^{8k}(E;\bfQ)$, as claimed. To see that the signature of $E$ vanishes, by the main result of \cite{ChernHirzebruchSerre}, it suffices  to show that $\pi_1(S^1\times S^1)$ acts trivially on $\oH^{*}(W_{1};\bfQ)$. This is clear in degree $0$ and $2n$ and follows in degree $n$ from the triviality of  $\widetilde{E}\ra S^1\times S^1$ as a fibration.

To prove the claim for a general fiber $M$ instead of $W_1$, we use the assumption $g(M)\ge1$ to write $M$ as a connected sum $M=W_1\sharp M'$ for a manifold $M'$, and observe that the $W_{1,1}$-bundle $\widetilde{E}\ra S^1\times S^1$ built above contains a trivial disk bundle in its bundle of boundaries, since we chose $\xi$ and $\eta$ to fix a neighborhood of $D^{4k}_-\subseteq S^{4k}$. Taking fiberwise boundary connected sum of $\widetilde{E}\ra S^1\times S^1$ with a trivial $M'\backslash \interior{D^{8k-2}}$ bundle over $S^1\times S^1$, followed by fiberwise coning of the boundary results in a $M$-bundle with the desired properties.
\end{proof}

\begin{thm}\label{thm:minimalPontryagyintopologically}For $k\neq1,2$, the Pontryagin number $p_k^2(E)\in\bfQ$ of the total space of a topological oriented bundle $\pi\colon E\ra S$ over a closed oriented surface with fiber a closed highly connected almost parallelizable $(8k-2)$-manifold $M$ is contained in the subgroup
\[2(p^{\min}_k)^2\cdot\bfZ\subseteq\bfQ.\] If $g(M)\ge1$, then there exists a bundle $\pi\colon E\ra S^1 \times S^1$ of the above type satisfying \[p_k^2(E)=2(p_k^{\min})^2\quad\text{and}\quad\sigma(E)=0.\]
\end{thm}
\begin{proof}Adapting the argument at the beginning of the proof of \cref{bigtheorem:bordismimage} in \cref{sect:proofofThmA} to the topological case by replacing $\BSO$ by $\BSTop$, one sees that $E$ is topologically oriented bordant to a topological $(4k-1)$-connected manifold, so the first part follows from \cref{lem:topologicallyevenform} ii). By \cref{lem:surjectivityofdiscbundles}, we can represent a generator of $\pi_{4k}\BTop$ by a class $\nu\in\pi_{4k}\BTop(4k-1)$ as long as $k\neq1,2$, from which the second part follows by applying \cref{lem:ConstructionBundleMinimalPontryagin} to the bundles $\nu=\xi=\eta$.
\end{proof}

\subsection{Divisibility of the signature}
We continue by proving \cref{bigthm:signaturestopologically}, which determines the values of the signature (or equivalently the corresponding Hirzebruch $\cL$-class) on the total spaces of the bundles in consideration. 

\begin{proof}[Proof of \cref{bigthm:signaturestopologically}]Arguing similarly as in the proof of \cref{thm:minimalPontryagyintopologically} and at the end of the proof of \cref{lem:ConstructionBundleMinimalPontryagin}, the claim follows from \cref{lem:topologicallyevenform} if we construct a $W_4=\sharp^4 (S^n\times S^n)$-bundle $E\ra S$ that contains a trivial $D^{4m}$-subbundle and satisfies $\sigma(E)=8$. More generally, we shall construct such a bundle with fiber $W_g=\sharp^g(S^n\times S^n)$ for any $g\ge4$. To this end, we fix a disk $D^{4m}\subseteq W_g$ and consider the commutative diagram
\begin{center}
\begin{tikzcd}[column sep=0.4cm]
\Bord_2(\BDiff(W_{g,1},D^{4m}))\arrow[d]\arrow[r,two heads]&\Bord_2(B\pi_0\Diff(W_{g,1},D^{4m}))\arrow[d]\arrow[r,two heads]&\Bord_2(\BSp^q_{2g}(\bfZ))\arrow[d,equal]\arrow[r,"\sigma"]&\bfZ\arrow[d,equal]\\
\Bord_2(\BHomeo(W_g,D^{4m}))\arrow[r,two heads]&\Bord_2(B\pi_0\Homeo(W_{g},D^{4m}))\arrow[r]&\Bord_2(\BSp^q_{2g}(\bfZ))\arrow[r,"\sigma"]&\bfZ,
\end{tikzcd}
\end{center} explained in the following. The manifold $W_{g,1}$ is obtained from $W_g$ by removing the interior of a disk that is disjoint from the previously chosen one. The two left vertical arrows are induced by coning off diffeomorphisms of $W_{g,1}$ to homeomorphisms of $W_g$. The two left horizontal arrows are induced by taking path-components, and are surjective since the canonical morphism $\oH_2(X;\bfZ)\ra\oH_2(B\pi_1(X);\bfZ)$ is surjective for any connected space $X$. The middle horizontal arrows are induced by the action on homology $\oH_n(W_{g,1};\bfZ)\cong\bfZ^{2g}$ preserving the (nondegenerate and antisymmetric) intersection form, so its image is contained in the symplectic group $\Sp_{2g}(\bfZ)\subseteq\GL_{2g}(\bfZ)$. For homeomorphisms that are isotopic to a diffeomorphism, this action lands in the subgroup $\Sp^q_{2g}(\bfZ)\subseteq\Sp_{2g}(\bfZ)$ of automorphisms which preserve Wall's quadratic refinement (see \cite{Wall2n}), which in the case of $W_g^{4m-2}$ with $m\neq1,2$ can be identified with the function $q\colon \bfZ^{2g}\ra \bfZ/2$ sending $\sum_{i=1}^g(a_ie_i+b_if_i)$ to $\sum_{i=1}^ga_ib_i$, where $\{e_i,f_i\}_{1\le i\le g}$ is the standard symplectic basis of $\bfZ^{2g}$ and $a_i,b_i\in\bfZ$. In this special case, the quadratic refinement turns out to also be preserved by homotopy equivalences (see e.g.\,\cite[Thm 10.3]{Baues}), so definitely by homeomorphisms, which explains the lower middle horizontal arrow. For the surjectivity of the upper right horizontal arrow, we use work of Wall \cite[Lem.\,23]{Wall_III}, \cite[Lem.\,10, 12]{Wall_II}, from which it follows that the homology action induces a short exact sequence
\[0\ra\Hom(\oH_n(W_{g}),S\pi_n\SO(n))\ra\pi_0\Diff(W_{g,1},D^{4m-2})\ra \Sp^q_{2g}(\bfZ)\ra 0,\] where $S\pi_n\SO(n)\subseteq\pi_n\SO(n+1)$ is the image of the stabilisation map $\pi_n\SO(n-1)\ra\pi_n\SO(n)$ and the induced $\Sp^q_{2g}(\bfZ)$-action on $\Hom(\oH_n(W_{g}),S\pi_n\SO(n))$ is the evident one.\footnote{Note that Wall considers pseudo-isotopy instead of isotopy classes, but these agree for simply connected manifolds by Cerf's ``pseudo-isotopy implies isotopy''.} The Serre spectral sequence of this extension, together with the fact that oriented bordism agrees with integral homology in low degrees, induces an exact sequence of the form
\[\Bord_2(B\pi_0\Diff(W_{g,1},D^{4m-2}))\ra \Bord_2(\BSp^q_{2g}(\bfZ))\ra \Hom(\oH_n(W_{g}),S\pi_n\SO(n))_{\Sp^q_{2g}(\bfZ)},\] so the surjectivity of $\Bord_2(B\pi_0\Diff(W_{g,1},D^{4m}))\ra\Bord_2(\BSp^q_{2g}(\bfZ))$ follows from showing that the coinvariants $\Hom(\oH_n(W_{g}),S\pi_n\SO(n))_{\Sp^q_{2g}(\bfZ)}\cong (\bfZ^{2g}\otimes S\pi_n\SO(n))_{\Sp^q_{2g}(\bfZ)}$ vanish, which one can see by acting with the matrices
\[\left(\begin{smallmatrix}P_\sigma&0\\0&P_\sigma\end{smallmatrix}\right)\in\Sp_{2g}^q(\bfZ)\quad\text{for }\sigma\in\Sigma_g,\quad\left(\begin{smallmatrix}0&-I_g\\I_g&0\end{smallmatrix}\right)\in\Sp_{2g}^q(\bfZ),\quad\text{and} \quad\left(\begin{smallmatrix}1&0&&&&\\1&1&&&&\\&&I_{g-2}&&&\\&&&1&-1&\\&&&0&1&\\&&&&&I_{g-2}\end{smallmatrix}\right)\in \Sp_{2g}^q(\bfZ),\]
where $P_\sigma\in\GL_g(\bfZ)$ denotes the permutation matrix associated to $\sigma\in\Sigma_g$ and $I_g\in\GL_g(\bfZ)$ is the unit matrix. Returning to the commutative diagram above, the rightmost morphism $\sigma\colon \Bord_2(\BSp^q_{2g}(\bfZ))\ra\bfZ$ is induced by assigning a bundle of symplectic lattices of rank $2g$ over a closed oriented surface its signature (see e.g.\,\cite[Lem 7.5]{GRWabelian}). By the argument of \cite{ChernHirzebruchSerre}, the horizontal composition $\Bord_2(\BHomeo(W_g,D^{4m}))\ra\bfZ$ agrees with the morphism that maps a topological bundle to the signature of its total space, so it suffices to show that its image contains $8\cdot\bfZ$. This is the case for the image of $\sigma\colon \Bord_2(\BSp_{2g}^q(\bfZ))\ra\bfZ$ as long as $g\ge4$ by the argument at the end of the proof of \cite[Thm 7.7]{GRWabelian}, so the surjectivity of $\Bord_2(\BDiff(W_g,D^{4m-2}))\ra \Bord_2(\BSp_{2g}^q(\bfZ))$ shows that also $\Bord_2(\BDiff(W_g,D^{4m-2}))$ contains a class of signature $8$ as long as $g\ge4$, and hence so does $\Bord_2(\BHomeo(W_g,D^{4m-2}))$, which finishes the proof.
\end{proof}

\subsection{Generalized Miller--Morita--Mumford classes, topologically}
Combining \cref{lem:homeodiffrationally} with \cref{bigthm:signaturestopologically}, we can conclude, similarly to the discussion in \cref{section:MMMclassessmoothly}, that the torsion free quotient of $\oH^2(\BHomeo(M);\bfZ)$ for a closed highly connected almost parallelizable $(4m-2)$-manifold $M$ with $g(M)\ge7$ and $m$ odd is generated by $\kappa_{\cL_m}/8$, or equivalently by $\frac{\sigma_m}{8}\cdot\frac{\kappa_{p_{m}}}{2(2m-1)!j_m}$. For $m=2k$ is even, \cref{thm:minimalPontryagyintopologically} together with \cref{bigthm:signaturestopologically} shows that the image of the morphism 
$(\sigma,p_k^2)\colon\Bord_2(\BHomeo(M))\ra\bfZ\oplus \bfQ$ is generated by $(8,0)$ and $(0,2{(p^{\min}_k)}^2)$, so $\oH^2(\BHomeo(M);\bfZ)_{\free}$ is generated by $\kappa_{\cL_{2k}}/8$ and $\kappa_{p_k^2}/(2{(p^{\min}_k)}^2)$. From this, after expressing $\cL_{2k}$ in terms of $p_{2k}$ and $p^2_{k}$ (see the beginning of \cref{section:charnumberformulas}), the basis claimed in \cref{bigthm:cobordismlatticeTOP} follows from the computation of ${p_k^{\min}}$ in \cref{lem:BrumfielMinimalPontryagin}.

\bibliographystyle{amsalpha}
\bibliography{literature}

\providecommand{\bysame}{\leavevmode\hbox to3em{\hrulefill}\thinspace}
\providecommand{\MR}{\relax\ifhmode\unskip\space\fi MR }
\providecommand{\MRhref}[2]{%
  \href{http://www.ams.org/mathscinet-getitem?mr=#1}{#2}
}
\providecommand{\href}[2]{#2}
\begin{thebibliography}{GMTW09}

\bibitem[ABK72]{AntonelliBurgheleaKahn}
P.~L. Antonelli, D.~Burghelea, and P.~J. Kahn, \emph{The non-finite homotopy
  type of some diffeomorphism groups}, Topology \textbf{11} (1972), 1--49.
  \MR{0292106}

\bibitem[Ada66]{Adams}
J.~F. Adams, \emph{On the groups {$J(X)$}. {IV}}, Topology \textbf{5} (1966),
  21--71. \MR{0198470}

\bibitem[AIK14]{BernoulliZeta}
T.~Arakawa, T.~Ibukiyama, and M.~Kaneko, \emph{Bernoulli numbers and zeta
  functions}, Springer Monographs in Mathematics, Springer, Tokyo, 2014, With
  an appendix by Don Zagier. \MR{3307736}

\bibitem[AK80]{AlexanderKahn}
J.~C. Alexander and S.~M. Kahn, \emph{Characteristic number obstructions to
  fibering oriented and complex manifolds over surfaces}, Topology \textbf{19}
  (1980), no.~3, 265--282. \MR{579576}

\bibitem[And69]{Anderson}
D.~R. Anderson, \emph{On homotopy spheres bounding highly connected manifolds},
  Trans. Amer. Math. Soc. \textbf{139} (1969), 155--161. \MR{0238332}

\bibitem[Ati69]{Atiyah}
M.~F. Atiyah, \emph{The signature of fibre-bundles}, Global {A}nalysis
  ({P}apers in {H}onor of {K}. {K}odaira), Univ. Tokyo Press, Tokyo, 1969,
  pp.~73--84. \MR{0254864}

\bibitem[Bau96]{Baues}
H.~J. Baues, \emph{On the group of homotopy equivalences of a manifold}, Trans.
  Amer. Math. Soc. \textbf{348} (1996), no.~12, 4737--4773. \MR{1340168}

\bibitem[BCS14]{BowdenCrowleyStipsicz}
J.~Bowden, D.~Crowley, and A.~I. Stipsicz, \emph{The topology of {S}tein
  fillable manifolds in high dimensions {I}}, Proc. Lond. Math. Soc. (3)
  \textbf{109} (2014), no.~6, 1363--1401. \MR{3293153}

\bibitem[BHS19]{BHS}
R.~Burklund, J.~Hahn, and A.~Senger, \emph{On the boundaries of highly
  connected, almost closed manifolds}, arXiv:1910.14116.

\bibitem[BL74]{BurgheleaLashofSmoothing}
D.~Burghelea and R.~Lashof, \emph{The homotopy type of the space of
  diffeomorphisms. {I}, {II}}, Trans. Amer. Math. Soc. \textbf{196} (1974),
  1--36; ibid. 196 (1974), 37--50. \MR{356103}

\bibitem[Bol12]{Boldsen}
S.~K. Boldsen, \emph{Improved homological stability for the mapping class group
  with integral or twisted coefficients}, Math. Z. \textbf{270} (2012),
  no.~1-2, 297--329. \MR{2875835}

\bibitem[Bro72]{Browder}
W.~Browder, \emph{Surgery on simply-connected manifolds}, Springer-Verlag, New
  York-Heidelberg, 1972, Ergebnisse der Mathematik und ihrer Grenzgebiete, Band
  65. \MR{0358813}

\bibitem[Bru68]{Brumfiel}
G.~Brumfiel, \emph{On the homotopy groups of {${\rm BPL}$} and {${\rm PL/O}$}},
  Ann. of Math. (2) \textbf{88} (1968), 291--311. \MR{0234458}

\bibitem[CHS57]{ChernHirzebruchSerre}
S.~S. Chern, F.~Hirzebruch, and J.-P. Serre, \emph{On the index of a fibered
  manifold}, Proc. Amer. Math. Soc. \textbf{8} (1957), 587--596. \MR{0087943}

\bibitem[GMTW09]{GMTW}
S.~Galatius, I.~Madsen, U.~Tillmann, and M.~Weiss, \emph{The homotopy type of
  the cobordism category}, Acta Math. \textbf{202} (2009), no.~2, 195--239.
  \MR{2506750}

\bibitem[Gol16]{Gollinger}
W.~Gollinger, \emph{{M}adsen\textendash{T}illmann\textendash{W}eiss spectra and
  a signature problem for manifolds}, PhD thesis, University of M{\"u}nster
  (2016).

\bibitem[GRW14]{GRWstable}
S.~Galatius and O.~Randal-Williams, \emph{Stable moduli spaces of
  high-dimensional manifolds}, Acta Math. \textbf{212} (2014), no.~2, 257--377.
  \MR{3207759}

\bibitem[GRW16]{GRWabelian}
\bysame, \emph{Abelian quotients of mapping class groups of highly connected
  manifolds}, Math. Ann. \textbf{365} (2016), no.~1-2, 857--879. \MR{3498929}

\bibitem[GRW17]{GRWII}
\bysame, \emph{Homological stability for moduli spaces of high dimensional
  manifolds. {II}}, Ann. of Math. (2) \textbf{186} (2017), no.~1, 127--204.
  \MR{3665002}

\bibitem[GRW18]{GRWI}
\bysame, \emph{Homological stability for moduli spaces of high dimensional
  manifolds. {I}}, J. Amer. Math. Soc. \textbf{31} (2018), no.~1, 215--264.
  \MR{3718454}

\bibitem[Har85]{Harer}
J.~L. Harer, \emph{Stability of the homology of the mapping class groups of
  orientable surfaces}, Ann. of Math. (2) \textbf{121} (1985), no.~2, 215--249.
  \MR{786348}

\bibitem[Hir66]{Hirzebruch}
F.~Hirzebruch, \emph{Topological methods in algebraic geometry}, Third enlarged
  edition. New appendix and translation from the second German edition by R. L.
  E. Schwarzenberger, with an additional section by A. Borel. Die Grundlehren
  der Mathematischen Wissenschaften, Band 131, Springer-Verlag New York, Inc.,
  New York, 1966. \MR{0202713}

\bibitem[Hir69]{HirzebruchSignature}
\bysame, \emph{The signature of ramified coverings}, Global {A}nalysis
  ({P}apers in {H}onor of {K}. {K}odaira), Univ. Tokyo Press, Tokyo, 1969,
  pp.~253--265. \MR{0258060}

\bibitem[KM63]{KervaireMilnor}
M.~A. Kervaire and J.~W. Milnor, \emph{Groups of homotopy spheres. {I}}, Ann.
  of Math. (2) \textbf{77} (1963), 504--537. \MR{0148075}

\bibitem[Kod67]{Kodaira}
K.~Kodaira, \emph{A certain type of irregular algebraic surfaces}, J. Analyse
  Math. \textbf{19} (1967), 207--215. \MR{0216521}

\bibitem[KS77]{KirbySiebenmann}
R.~C. Kirby and L.~C. Siebenmann, \emph{Foundational essays on topological
  manifolds, smoothings, and triangulations}, Princeton University Press,
  Princeton, N.J.; University of Tokyo Press, Tokyo, 1977, With notes by John
  Milnor and Michael Atiyah, Annals of Mathematics Studies, No. 88.
  \MR{0645390}

\bibitem[Lam81]{Lampe}
R.~Lampe, \emph{Diffeomorphismen auf {S}phären und die {M}ilnor-paarung},
  Diploma Thesis, Johannes Gutenberg University Mainz (1981).

\bibitem[Lev85]{Levine}
J.~P. Levine, \emph{Lectures on groups of homotopy spheres}, Algebraic and
  geometric topology ({N}ew {B}runswick, {N}.{J}., 1983), Lecture Notes in
  Math., vol. 1126, Springer, Berlin, 1985, pp.~62--95. \MR{802786}

\bibitem[LM89]{LawsonMichelsohn}
H.~B. Lawson, Jr. and M.-L. Michelsohn, \emph{Spin geometry}, Princeton
  Mathematical Series, vol.~38, Princeton University Press, Princeton, NJ,
  1989. \MR{1031992}

\bibitem[Mey72]{MeyerThesis}
W.~Meyer, \emph{Die {S}ignatur von lokalen {K}oeffizientensystemen und
  {F}aserb\"undeln}, Bonn. Math. Schr. (1972), no.~53, viii+59. \MR{0305402}

\bibitem[Mey73]{Meyer}
\bysame, \emph{Die {S}ignatur von {F}l\"achenb\"undeln}, Math. Ann.
  \textbf{201} (1973), 239--264. \MR{0331382}

\bibitem[MK60]{KervaireMilnorBernoulli}
J.~W. Milnor and M.~A. Kervaire, \emph{Bernoulli numbers, homotopy groups, and
  a theorem of {R}ohlin}, Proc. {I}nternat. {C}ongress {M}ath. 1958, Cambridge
  Univ. Press, New York, 1960, pp.~454--458. \MR{0121801}

\bibitem[MW07]{MadsenWeiss}
I.~Madsen and M.~Weiss, \emph{The stable moduli space of {R}iemann surfaces:
  {M}umford's conjecture}, Ann. of Math. (2) \textbf{165} (2007), no.~3,
  843--941. \MR{2335797}

\bibitem[Ped75]{Pedersen}
E.~K. Pedersen, \emph{Embeddings of topological manifolds}, Illinois J. Math.
  \textbf{19} (1975), 440--447. \MR{0377898}

\bibitem[Qui71]{Quillen}
D.~Quillen, \emph{The {A}dams conjecture}, Topology \textbf{10} (1971), 67--80.
  \MR{0279804}

\bibitem[Rov18]{Rovi}
C.~Rovi, \emph{The nonmultiplicativity of the signature modulo 8 of a fibre
  bundle is an {A}rf-{K}ervaire invariant}, Algebr. Geom. Topol. \textbf{18}
  (2018), no.~3, 1281--1322. \MR{3784006}

\bibitem[RS70]{RourkeSanderson}
C.~P. Rourke and B.~J. Sanderson, \emph{On topological neighbourhoods},
  Compositio Math. \textbf{22} (1970), 387--424. \MR{0298671}

\bibitem[RW16]{RWresolutions}
O.~Randal-Williams, \emph{Resolutions of moduli spaces and homological
  stability}, J. Eur. Math. Soc. (JEMS) \textbf{18} (2016), no.~1, 1--81.
  \MR{3438379}

\bibitem[RW17]{RWnote}
\bysame, \emph{A fibre bundle of signature 4}, available on the author's
  website.

\bibitem[Sch72]{Schultz}
R.~Schultz, \emph{Composition constructions on diffeomorphisms of {$S^{p}\times
  S^{q}$}}, Pacific J. Math. \textbf{42} (1972), 739--754. \MR{0315731}

\bibitem[Ste75]{Stern}
R.~J. Stern, \emph{On topological and piecewise linear vector fields}, Topology
  \textbf{14} (1975), no.~3, 257--269. \MR{0394659}

\bibitem[Sto87]{Stolz}
S.~Stolz, \emph{A note on the {$bP$}-component of {$(4n-1)$}-dimensional
  homotopy spheres}, Proc. Amer. Math. Soc. \textbf{99} (1987), no.~3,
  581--584. \MR{875404}

\bibitem[Wal62]{Wall2n}
C.~T.~C. Wall, \emph{Classification of {$(n-1)$}-connected {$2n$}-manifolds},
  Ann. of Math. (2) \textbf{75} (1962), 163--189. \MR{0145540}

\bibitem[Wal63]{Wall_II}
\bysame, \emph{Classification problems in differential topology. {II}.
  {D}iffeomorphisms of handlebodies}, Topology \textbf{2} (1963), 263--272.
  \MR{0156354}

\bibitem[Wal65]{Wall_III}
\bysame, \emph{Classification problems in differential topology. {III}.
  {A}pplications to special cases}, Topology \textbf{3} (1965), 291--304.
  \MR{0177421}

\bibitem[Wal67]{Wall2n+1}
\bysame, \emph{Classification problems in differential topology. {VI}.
  {C}lassification of {$(s-1)$}-connected {$(2s+1)$}-manifolds}, Topology
  \textbf{6} (1967), 273--296. \MR{0216510}

\end{thebibliography}
\vspace{0.15cm}

\end{document}